\documentclass[11pt,twoside]{amsart}
\usepackage{amsthm,amsfonts,amssymb,amscd,euscript,amsmath,mathrsfs,mathtools}
\usepackage{appendix,enumerate, accents}
\usepackage{tikz-cd}
\usepackage[colorlinks=true, pdfstartview=FitV, linkcolor=black,citecolor=black,urlcolor=black]{hyperref}

\newenvironment{red}{\relax\color{red}}{\relax}
\newenvironment{blue}{\relax\color{blue}}{\hspace*{.5ex}\relax}
\newcommand{\ber}{\begin{red}}
\newcommand{\er}{\end{red}}
\newcommand{\beb}{\begin{blue}}
\newcommand{\eb}{\end{blue}}

\def\A{\mathcal{A}}

\def\C{\mathcal{C}}
\def\H{\mathcal{H}}
\def\P{\mathcal{P}}
\def\Q{\mathcal{Q}}
\def\R{\mathcal{R}}
\def\X{\mathcal{X}}
\def\Z{\mathcal{Z}}

\def\CC{\mathbb{C}}
\def\ZZ{\mathbb{Z}}
\newcommand{\QQ}{\mathbb{Q}}
\def\L{\mathbb{L}}
\def\N{\mathbb{N}}
\def\F{\mathbb{F}_q}

\def\U{\mathbf{U}}
\def\m{\mathbf{m}}

\def\g{\mathfrak{g}}

\def\e{\mathrm{e}}
\def\h{\mathrm{h}}

\def\Zt{{\mathbb{Z}_2}}

\newcommand{\cplx}[4]{\begin{tikzcd}[cramped, column sep=scriptsize] 
#1 \ar[r, "#2", shift left] \pgfmatrixnextcell  #4 \ar[l,  "#3", shift left] 
\end{tikzcd}}

\newcommand{\Hom}{\textup{Hom}}
\newcommand{\Ext}{\textup{Ext}}
\newcommand{\End}{\textup{End}}

\newcommand{\ostar}{\circledast}

\newcommand{\Ho}{\operatorname{Ho}}
\newcommand{\Ro}{\operatorname{Rt}}

\newcommand{\D}{\operatorname{D}}
\newcommand{\lra}{\longrightarrow}

\newcommand{\blob}{{\scriptscriptstyle\bullet}}

\newcommand{\isom}{\cong}
\renewcommand{\ker}{\operatorname{ker}}
\newcommand{\im}{\operatorname{im}}
\newcommand{\cok}{\operatorname{coker}}
\newcommand {\Aut}{\operatorname{Aut}}

\newcommand {\<}{\langle}
\renewcommand {\>}{\rangle}
\renewcommand{\dag}{*}
\newcommand{\fin}{\mathsf{fin}}

\newcommand{\rd}{\mathsf{red}}
\renewcommand{\DH}{\operatorname{\mathcal{DH}}}

\newcommand{\diam}{\diamond}

\renewcommand{\k}{\Bbbk}
\newcommand{\bc}{\boldsymbol{k}}
\newcommand{\cl}[1]{{\hat #1}}
\newcommand{\gen}[1]{[#1]}
\newcommand{\lRa}[1]{\stackrel{#1}{\lra}}

\newcommand{\isoto}{{\overset{\sim}{\rightarrow}}}
\newcommand {\id}{\operatorname{id}}
\newcommand{\longinto}{\lhook\joinrel\longrightarrow}
\newcommand{\Iso}{\operatorname{Iso}}
\newcommand{\tensor}{\otimes}

\theoremstyle{plain}
\newtheorem{thm}{Theorem}[section]
\newtheorem{lem}[thm]{Lemma}
\newtheorem{prop}[thm]{Proposition}
\newtheorem{cor}[thm]{Corollary}

\theoremstyle{definition}
\newtheorem{definition}[thm]{Definition}
\newtheorem{example}[thm]{Example}
\newtheorem{remark}[thm]{Remark}

\numberwithin{equation}{section} 
\numberwithin{figure}{section}
\numberwithin{table}{section}

\begin{document}
\title[Quantum GKM algebras via Hall algebras of complexes]{Quantum generalized Kac--Moody algebras\\ via Hall algebras of complexes} 

\author[J. D. Axtell]{Jonathan D. Axtell$^{\dagger}$}
\thanks{$^{\dagger}$This paper was supported by the National Research Foundation of Korea (NRF) funded by the Ministry of Science (NRF-2017R1C1B5018384).} 
\address{Sungkyunkwan University, Suwon 16419, Republic of Korea} 
\email{jaxtell@skku.edu}

\author[K.-H. Lee]{Kyu-Hwan Lee$^{\ast}$}
\thanks{$^{\ast}$This work was partially supported by a grant from the Simons Foundation (\#318706).}
\address{Department of Mathematics, University of Connecticut, Storrs, CT 06269, U.S.A.}
\email{khlee@math.uconn.edu}


\begin{abstract}
We establish an embedding of the quantum enveloping algebra of a symmetric generalized Kac--Moody algebra into a localized Hall algebra of  $\mathbb Z_2$-graded complexes of
representations of a quiver with (possible) loops.  
To overcome difficulties resulting from the existence of infinite dimensional projective objects, we consider the category of finitely-presented representations and
the  category of $\mathbb Z_2$-graded complexes of projectives with finite homology.
\end{abstract}

\maketitle

\section{Introduction} \label{sec-int}

Let $\A$ be 
an abelian category 
such that the sets $\Hom(A,B)$ and $\Ext^1(A,B)$ are both finite for all $A,B\in \A$. 
The {\em Hall algebra} of $\A$  is defined to be the $\CC$-vector space with 
basis elements indexed by isomorphism classes in $\A$ 
and with associative multiplication which encodes information about extensions of objects. 
Typical examples of such
abelian categories arise as the category $\mathrm{rep}_\k(\mathcal Q)$ of finite-dimensional representations of an acyclic  quiver $\mathcal Q$ over a finite field $\k:=\mathbb F_q$. This category became a focal point of intensive research when C. Ringel \cite{Ri} realized one half of a quantum group via a twisted Hall algebra  of the category. This twisted Hall algebra is usually called the {\em Ringel--Hall} algebra. The construction was further generalized by J.~A. Green \cite{Gr} to one half of the quantum group of an arbitrary Kac--Moody algebra. 

Even though there is a construction, called {\em Drinfeld double}, which glues together two copies of one-half quantum group to obtain the whole quantum group, it is desirable to have an explicit realization of the whole quantum group in terms of a Hall algebra. 
Among various attempts, the idea of using a category of $\mathbb Z_2$-graded complexes was suggested by the works of M. Kapranov \cite{Kapra}, L. Peng and J. Xiao \cite{PX,PX1}. In his seminal work \cite{Br}, T. Bridgeland successfully utilized this idea to achieve a Hall algebra realization of  the whole quantum group. More precisely, given a Kac--Moody algebra $\mathfrak g$, he took the category $\mathrm{rep}_\k(\mathcal Q)$  of finite dimensional representations of an acyclic  quiver $\mathcal Q$ associated with $\mathfrak g$, and considered the full subcategory $\P$ of projective objects in $\mathrm{rep}_\k(\mathcal Q)$. By studying the category $\mathcal C(\P)$ of $\mathbb Z_2$-graded complexes in $\P$, he showed that the whole quantum group is embedded into the reduced localization of a twisted Hall algebra of $\mathcal C(\P)$.

The purpose of this paper is to extend Bridgeland's construction to {\em generalized Kac--Moody algebras}. These algebras were introduced by R. Borcherds  \cite{Bor} around 1988. He used a  generalized Kac--Moody algebra, called the {\em Monster Lie algebra},  to prove the celebrated Moonshine Conjecture \cite{Bor-1}. Since then, many of the constructions in the theory of Kac--Moody algebras have been extended to generalized Kac--Moody algebras. In particular, the quantum group of a generalized Kac--Moody algebra was defined by Kang
\cite{Ka95}, and one half of the quantum group was realized via a Hall algebra by Kang and Schiffmann \cite{KS}, following Ringel--Green's construction.  

The main difference from the usual Kac--Moody case is that the quiver $\mathcal Q$ may have loops in order to account for imaginary simple roots. 
A natural question arises:
\begin{quote}
{\em Is it possible to realize the whole quantum group of a generalized Kac--Moody algebra  in terms of a Hall algebra of $\ZZ_2$-graded complexes?} 
\end{quote}
A straightforward approach would run into an obstacle. Namely, when there is a loop, a projective object may well be infinite dimensional, and the Hall product would not be defined.

In this paper, we show that this difficulty can be overcome by considering the category $\R$ 
of finitely-presented representations 
of a locally finite quiver $\Q$, possibly with loops,  
and the category $\C_\fin(\P)$ 
 of $\mathbb Z_2$-graded complexes in the category of projectives $\P \subset \R$  with finite homology.
In contrast to Bridgeland's construction, however, it is not clear that the 
corresponding product in the Hall algebra $\H(\C_\fin(\P))$
 is associative.  Therefore the proof of associativity for the (localized) Hall algebra  is one of the main results of this paper.

Following the approach of \cite{Br}, we work in the more general setting of a category $\R$ satisfying certain natural assumptions listed in  the  next subsection (Section \ref{ass}) which we keep throughout Sections \ref{sec-q-h}-\ref{sec:dbl}. 
The main theorem (Theorem \ref{thm:Drinf}) for this general setting states that a certain localization $\DH(\R)$ 
of the  Hall algebra  $\H(\C_\fin(\P))$ 
is isomorphic to the Drinfeld double of the (extended) Hall algebra of $\R$,  generalizing a result of Yanagida \cite{Ya}. 
As a corollary (Corollary \ref{cor-associative}), the localized Hall algebra  is shown to be an associative algebra.

In Section \ref{sec-q-group}, we show that the category $\R$ of finitely-presented representations of a locally finite quiver satisfies all the assumptions in Section \ref{ass} under some minor restrictions on the quiver. As a consequence of the results of Section \ref{sec:dbl} in the general setting, we obtain the main result for the quantum group (Theorem \ref{thm-main-q-group}) which establishes an embedding 
\[\Xi \colon \U_v \longinto\DH_{\rd}(\R)\]
of the whole quantum group $\U_v=\U_v(\g)$ of a generalized Kac--Moody algebra $\g$ into a reduced version of the localized Hall algebra $\DH(\R)$. 

\subsection{Assumptions}
\label{ass}
Given an abelian category $\C$, let $K(\C)$ denote its Grothendieck group  
and write $\bc_\C(X) \in K(\C)$ to denote the class of an object $X\in \C$. 

Throughout this paper $\R$ is an  
abelian category. 
Let $\P\subset \R$ denote the 
full subcategory of projectives  
and $\A \subset \R$ the full subcategory 
of objects $A\in \R$ such that $\Hom_\R(M,A)$ is a finite set for any $M\in \R$. 
There are several conditions that we will impose on 
the triple $(\R, \P, \A)$. 
Precisely, we shall always assume that 
\begin{itemize}
\item[(a)] $\R$ is essentially small and linear over  $\k=\F$,
\item[(b)] $\R$ is hereditary, that is of global dimension at most 1, and has enough projectives, 
\item[(c)] for any objects $P,Q,M\in \P$, the relation $M\oplus P \cong M\oplus Q$ implies $P\cong Q$, 
\item[(d)] every element in $K(\R)$ is a $\QQ$-linear combination of elements in $\{\bc_\R(A)\mid A\in \A\}$, 
\item[(e)] the identity $\bc_\R(A)=\bc_\R(B)$ implies  $| \Hom(P,A)|=|\Hom(P,B)|$ for all $P \in \P$, $A, B \in \A$. 
\end{itemize}

It is clear that the category $\A$ is Hom-finite, 
that is $\Hom_\R(A,B)$ is a finite set for all $A,B\in\A$.  
Since $\R$ has enough projectives by (b), 
it follows that the subcategory $\A$ is abelian 
and hence a Krull--Schmidt category by \cite{Kr}.
It is also easy to check that $\A$ is closed under extensions in $\R$. 
We note, however, that the category $\R$ is not necessarily Krull-Schmidt in general. 

The condition (c) is required in the proof of Proposition \ref{free} to show that 
the localized Hall algebra $\DH(\R)$ is a free module over the 
group algebra $\CC[K(\R)\times K(\R)]$. 
Conditions (d) and (e) are needed in Section \ref{ss:Euler} 
to ensure that the Euler form on $\A$ 
can be lifted to a bilinear form 
on the whole category $\R$, albeit with values in $\QQ$. 

The class of categories $\R$ 
satisfying the above conditions (a)-(e) generalizes the class of  
categories $\A$ satisfying Bridgeland's conditions, also denoted (a)-(e) in \cite{Br}, 
 although our conditions do not correspond precisely. 
In particular, if the category $\R$ is Hom-finite then $\A=\R$ so that condition (d) is superfluous, 
and it is also clear that (c) holds since $\R$ is Krull-Schmidt in this case. 
One may check using Proposition \ref{prop-nonzero}
 that the remaining conditions (a), (b), (e) hold 
in case $\A=\R$ if and only if the conditions (a)-(e) in \cite{Br} hold for $\A$.  

\subsection{Notation} \label{sec-notations}
Assume throughout that $\k=\F$ is a finite field ($q \ge 2$) with $q$ elements.  Let $v \in \mathbb R_{>0}$ be such that $v^2 =q$. We denote by $v^{1/n}$ the positive real $2n$-th root of $q$ for $n \in \ZZ_{\ge 1}$, and write $q^{1/n}=v^{2/n}$. We also write $\ZZ_2=\ZZ/2 \ZZ$.

\section{Hall algebras} \label{sec-q-h}

We assume that  the triple $(\R,\P,\A)$ 
satisfies axioms (a)-(e) in Section \ref{ass}. 

\subsection{Hall algebras}
Given a small category $\C$, denote by  $\Iso(\C)$ the set of isomorphism classes in $\C$. 
Suppose that $\C$ is abelian. Given objects $X,Y,Z\in \C$, define $\Ext^1_\C(X,Y)_Z$ to be the set of 
(equivalence classes of) extensions with middle term isomorphic to $Z$ in $\Ext^1_\C(X,Y)$. 

Since $\R$ has enough projectives, it follows from the definition of $\A$ 
that the set $\Ext^1_\A(A,B)$ is finite for all $A, B \in \A$.
The {\em Hall algebra} $\H(\A)$ 
is defined to be the $\CC$-vector space with basis 
indexed by elements $A\in\Iso(\A)$, 
and with associative multiplication defined by 
\begin{equation}
\label{defprod}
\gen{A}\diam \gen{B}= 
\sum_{C \in\Iso(\A)} \frac{|\Ext^1_\A(A,B)_{C}|}{|\Hom_\A(A,B)|} \ \gen{C}.\end{equation}
The unit is given by $\gen{0}$, where $0$ is the zero object in $\A$. 

Recall from \cite{Br} that the multiplication \eqref{defprod}  is a variant of the usual Hall product 
(see e.g.~\cite{Ri}) defined as follows. 
Given objects $A,B,C\in \A$, define the number
\begin{equation} \label{eqn-gABC}
g_{A,B}^C=\big|\big\{ B'\subset C :  B'\isom B,\ C/B'\isom A\big\}\big|.\end{equation}
Writing $a_A = |\Aut_\A(A)|$ to denote the cardinality of the automorphism group of an object $A$, 
recall that 
\[g_{A,B}^C=\frac{|\Ext^1_\A(A,B)_{C}|}{|\Hom_\A(A,B)|}\cdot\frac{a_C}{a_A \, a_B}.\]
Hence using  
$[[A]]=\gen{A}\cdot a_A^{-1}$ 
as alternative generators,  
the product takes the form
\[[[A]]\diam [[B]]= 
\sum_{C\in\Iso(\A)} g_{A,B}^C\cdot [[C]].\]
The associativity of multiplication 
in $\H(\A)$
then reduces to the equality
\begin{equation} \label{eq:assoc}
\sum_{C_1} g_{A,B}^{C_1}\,  g^{D}_{C_1,C}
\ =\ 
\sum_{D_1} g_{A,D_1}^{D}\,  g^{D_1}_{B,C}
\end{equation}
which holds for any $A,B,C,D\in \Iso(\A)$.

\subsection{Twisted and Extended Hall algebras}\label{ss:Extend}
 
The Euler form on  $\A$ is a bilinear mapping 
\[
\<-,-\> : \A \times \A \to \ZZ
\]
defined by 
\begin{equation}
\label{Euler}
 \< A,B\> := \dim_{\k}\Hom(A,B) - \dim_{\k}\Ext^1(A,B) 
 \end{equation} 
for all $A,B\in \A$. 
This form factors through the Grothendieck group $K(\A)$ (see Lemma \ref{lem-gr-well}). 
We also  introduce the associated symmetric form
$(A,B)=\langle A,B \rangle + \langle B,A\rangle.$

We define the {\em twisted Hall algebra} $\H_v(\A)$ to be the same as the Hall algebra of $\A$ with  a new multiplication given by
\begin{equation} \label{eqn-ttw} \gen{A} * \gen{B}
=v^{\< \bc_\A(A),\, \bc_\A(B)\rangle} \gen{A}\diam \gen{B}.\end{equation}
We extend $\H_v(\A)$ by  adjoining 
new generators $K_\alpha$ for all $\alpha\in K(\A)$
with  the relations 
\begin{equation} 
\label{extend}
K_\alpha * K_\beta = K_{\alpha+ \beta}, \qquad 
K_\alpha * \gen{A}*K_{-\alpha} =\  
v^{(\alpha,\, \bc_\A(A))}\,  \gen{A} ,
\end{equation}
where we use the symmetric form $(\, , \, )$.
The resulting algebra will be called  the {\em extended Hall algebra} and denoted by $\widetilde{\H_v}(\A)$.

\subsection{Generalized Euler form}\label{ss:Euler}

Since the category $\R$ has enough projectives, it follows from the definition of $\A$ that $\Ext^1_\R(M,A)$ is a finite set 
for any $M\in \R$ and $A\in \A$. 
Define 
an Euler form $\< -, - \>: \R \times \A \to \mathbb{Z}$ 
 by setting 
\[
\< M, A \> := \mathrm{dim}_\k\Hom_\R(M,A) - \mathrm{dim}_\k\Ext^1_\R(M,A)
\]
for all $M\in \R, A\in \A$.  
The following lemma shows that the Euler form 
induces a bilinear map 
\[
\< -, - \>: K(\R) \times K(\A) \to \mathbb{Z}. 
\]

\begin{lem} \label{lem-gr-well}
The Euler form $\< M, A\>$
depends only on the classes of objects $M\in \R$ and  $A\in \A$ in the Grothendieck groups $K(\R)$
and $K(\A)$, respectively.
\end{lem}

\begin{proof}
Suppose $0 \to M' \to M \to M'' \to 0$ is an exact sequence in $\R$.  Then there is a long 
exact sequence 
\begin{align*}
0 \to \Hom(M'', A) \to \Hom(M, A) \to \Hom(M', A)\\[0.2cm]
\to \Ext(M'',A) \to \Ext(M,A) \to \Ext(M',A) \to 0
\end{align*}
which shows that 
\[\< M'',A\> - \< M,A\> + \< M',A\> = 0.\] 
So the Euler form is well-defined on the class of $M$.  The proof for the class of $A$ is similar. 
\end{proof}

In the remainder, let us write $\cl{X} = \bc_\R(X)$ to denote the class of 
an object $X$ in $K(\R)$, and continue to write $\bc_\A(-)$ for classes 
in $K(\A)$. 
The inclusion $\A\subset \R$
 induces a canonical map 
\begin{equation} \label{eqn-kr-ka}
K(\A) \to K(\R), \quad 
\bc_\A(A) \mapsto \cl{A}. 
\end{equation}
Write $\bar{K}(\A) \subset K(\R)$ to denote the image of $K(\A)$ under this map.

Suppose that $\{P_i\}_{i\in I}$ is a complete list of isomorphism classes of
indecomposable projective objects in $\P$, for some indexing set $I$. 
Then define the {\em dimension vector} of an object $A\in \A$ to be 
\[\mathbf{dim}\, A := (\dim_\k \Hom_\R(P_i,A))_{i\in I} \in \mathbb{Z}^{\oplus I} .\]
From condition (e), we have 
\[
\mathbf{dim}\, A = \mathbf{dim}\, B
\] 
for any objects 
$A, B\in \A$, whenever the classes $\hat{A} = \hat{B}$ are equal in the Grothendieck group $K(\R)$.

\begin{lem} \label{lem-mama}
Suppose that $A, A' \in \A$. 
Then $\< M,A \> = \<M,A'\>$ for all $M\in \R$ if and only if 
\[\dim_\k\Hom(P,A) = \dim_\k\Hom(P,A')\]
for all $P\in \P$. 
\end{lem}
\begin{proof}
The only-if-part  follows from the fact that 
$\<P,A\> = 
\dim_\k\Hom(P,A)$ for all $P\in \P$ and $A\in \A$.
For the if-part, let $M\in \R$, $A\in \A$, and suppose 
$0 \to P \to Q \to M \to 0$ is a projective resolution. 
Then there is a long exact sequence 
\[
0 \lra \Hom_\R(M, A) 
\lra \Hom_\R(Q, A)
\lra \Hom_\R(P,A)
\lra \Ext_\R(M,A)
\lra 0,\]
which shows that 
\[
\<M,A\> = \dim_\k\Hom(Q,A) - \dim_\k\Hom(P,A).
\]
The converse statement now follows. 
\end{proof}

The above lemma now has the following consequence.

\begin{cor} \label{cor-fac-th}
The Euler form on $\R \times \A$ 
factors through a bilinear form 
\[ \< -, - \>: K(\R) \times \bar{K}(\A) \to \mathbb{Z}.\]
\end{cor}

Now suppose $X,Y\in \R$.  Then it follows from condition (d) that 
there exist objects 
$A_1, \dots, A_r \in \A$ such that 
\[\cl{Y} = 
c_1\cl{A}_1 + \dots + c_r\cl{A}_r  \]
for some coefficients $c_i \in \QQ$. 
Define a bilinear form 
\[
K(\R)\times K(\R) \to \QQ
\] 
by setting 
\begin{equation}\label{Euler2}
\< \cl{X},\cl{Y}\> = \sum_i  c_i  \cdot \<\cl{X}, \cl{A}_i\>.
\end{equation}
and extend it through linearity.
We again define a symmetric version by setting 
$(\cl{X},\cl{Y}) = \<\cl{X},\cl{Y}\>+\<\cl{Y},\cl{X}\>$ 
for any $X,Y\in \R$. 

\begin{prop} \label{prop-nonzero}
Let $A$ be a nonzero object in $\A$. Then the classes $\bc_\A(A) \in K(\A)$ and $\cl{A} \in K(\R)$ are both nonzero. 
\end{prop}

\begin{proof}
 Consider a nonzero object $A \in \A$. Since $\R$ has enough projectives, there exists a surjection $P\rightarrow A \rightarrow 0$ for some projective $P \in \P$.  It follows that $\Hom(P,A)$ is a nonzero set. We obtain
\[   \<P, A\>  =  \dim_\k \Hom(P,A) - \dim_\k \Ext(P,A)  =  \dim_\k \Hom(P,A)   \geq  1.\]
By Lemma \ref{lem-mama} and Corollary \ref{cor-fac-th}, we obtain
\[   \<P, A\> =\<\cl{P}, \bc_\A(A) \>  = \<\cl{P}, \cl{A}\>  \geq  1.\]Thus neither  $\cl{A}$ nor $\bc_\A(A)$ is zero.  
\end{proof}

Denote by ${K}_{\geq 0}(\A)\subset K(\A)$ the positive cone in the  Grothendieck group generated by the classes $\bc_\A(A)$ for  $A\in\A$. 
Define
\begin{equation} \label{order}
\alpha\leq \beta \iff \beta-\alpha\in {K}_{\geq 0}(\A).
\end{equation}
Then it follows from Proposition \ref{prop-nonzero} that $\leq$ is a partial order on $K(\A)$.

\subsection{The extended Hall algebra $\widetilde{\H_v}(\R)$}\label{ss:Extend-1}

The algebra $\H_v(\A)$ is naturally graded by the Grothendieck group 
$K(\R)$: 
\begin{equation*}
\H_v(\A) = \bigoplus_{\alpha \in K(\R)} \H_v(\A)_{(\alpha)}, \qquad
\H_v(\A)_{(\alpha)}  :=  \bigoplus_{\cl{A}=\alpha} \CC\gen{A}.
\end{equation*}

We define a slight modification of the extended Hall algebra 
 $\widetilde{\H_v}(\A)$
from Section \ref{ss:Extend}. 
Starting again from $\H_v(\A)$ with multiplication \eqref{eqn-ttw}, 
the {\em extended Hall algebra} 
 $\widetilde{\H_v}(\R)$
 is defined by adjoining  generators $K_\alpha$ for all $\alpha \in K(\R)$ 
with  the relations 
\begin{equation}
\label{extend2}
K_\alpha \ast K_\beta = K_{\alpha+\beta}, \qquad 
K_\alpha * \gen{A}*K_{-\alpha} =\  
v^{(\alpha, \cl{A})} \gen{A}.  
\end{equation}
The algebra $\widetilde{\H_v}(\R)$ is also $K(\R)$-graded, 
with the degree of each $K_\alpha$ equal to zero. 
Note that the multiplication map 
\[\H_v(\A)  \otimes_\CC \CC[K(\R)] 
\to \widetilde{\H_v}(\R)\]
is an isomorphism of vector spaces.

Following Green \cite{Gr} and Xiao \cite{Xi}, 
we define a coalgebra structure on $\widetilde{\H_v}(\R)$. 
We refer to \cite{Ya} for the definition of a {\em topological} coalgebra,  
which involves a completed tensor product.

\begin{definition}\cite{Gr,Xi}
In the extended Hall algebra $\H=\widetilde{\H_v}(\R)$ 
define $\Delta:\H \to \H\,\widehat{\otimes}_{\CC}\, \H$, and $\epsilon: \H \to \CC$, 
by setting
\begin{align*}
\label{eq:Green}
 \Delta([A] K_\alpha) :=  \sum_{B,C \in \Iso(\A)} v^{\<B,C\>} g^A_{B,C}  \cdot  (\gen{B} K_{\cl{C} + \alpha}) \otimes (\gen{C} K_{\alpha}), \qquad
 \epsilon([A] K_\alpha):=\delta_{A,0}
\end{align*} 
for all $A\in \Iso(\A)$, $\alpha \in K(\R)$, 
where  the numbers $g^A_{B,C}$ are defined in \eqref{eqn-gABC} and   
$\H \,\widehat{\otimes}_{\CC}\, \H$ is the completed tensor product.
This gives $\widetilde{\H_v}(\R)$ the structure of a topological coassociative coalgebra. 
\end{definition}

As noted in \cite[Remark 1.6]{Sch}, 
the coproduct $\Delta$ on $\H=\widetilde{\H_v}(\R)$ takes values in 
$\H \otimes \H$, instead of the completion
$\H \,\widehat{\otimes}\, \H$, 
if and only if the following condition holds:
\begin{equation}\label{sub}
\text{\em Any fixed object $A \in \A \subset \R$ has only finitely many subobjects $B\subset A$}. 
\end{equation}
This condition is satisfied if $\R$ is the category of quiver representations 
considered in Section \ref{sec-q-group}. 

Now we have an algebra structure 
and a coalgebra structure 
on  $\widetilde{\H_v}(\R)$. 
It follows from \cite{Gr,Xi} that  these structures are compatible to give a (topological) \emph{bialgebra} structure.
Below, we simply denote by 
$\widetilde{\H_v}(\R)$
the bialgebra
$(\widetilde{\H_v}(\R),\ast,[0],\Delta,\epsilon)$.
The bialgebra  $\widetilde{\H_v}(\R)$ 
admits a natural bilinear form compatible with the bialgebra structure 
called a \emph{Hopf pairing}.

\begin{definition}[\cite{Gr,Xi, Ya}]
\label{df:hp} 
Define a  bilinear form $(\cdot,\cdot)_H$ on $\widetilde{\H_v}(\R)$ 
by setting 
\[
 ([A] K_\alpha,[B] K_\beta)_H 
  := 
  v^{(\alpha,\beta)} a_{A}
\,\delta_{A,B}
\]
for $\alpha , \beta \in K(\R)$, where 
$a_A = |\Aut_{\mathcal{A}}(A)|$ as before. 
\end{definition}

It is clear that the restriction of this bilinear form to the 
subalgebra $\H_v(\A) \subset \widetilde{\H_v}(\R)$ 
is nondegenerate. 
The following result was stated for 
$\widetilde{\H_v}(\A)$ in \cite{Xi}.  
It is easy to check that it holds for $\widetilde{\H_v}(\R)$ as well. 

\begin{prop}[\cite{Sch,Xi}] \label{prop-Hopf-pairing}
The bilinear form $(\cdot,\cdot)_H$ is
a  Hopf pairing 
on the bialgebra $\widetilde{\H_v}(\R)$,
that is, 
for any $x,y,z \in  \widetilde{\H_v}(\R)$, 
one has 
\begin{equation*}
   (1, x)_H
 = \epsilon(x), \qquad
   (x * y, z)_H
 = (x \otimes y, \Delta(z))_H
\end{equation*}
where
we use the usual pairing on the tensor product space:
\[
 (x \otimes y, z \otimes w)_H = 
 (x,z)_H \, (y, w)_H .
\]
\end{prop}

\subsection{The Drinfeld double}

We briefly recall the Drinfeld double construction
for Hall bialgebras. 
A complete treatment of Drinfeld doubles is given in 
\cite[\S 3.2]{Jo} and \cite[\S 5.2]{Sch}. 

In \cite{Xi}, Xiao showed that the extended Hall algebra 
$\widetilde{H_v}(\A)$ is a Hopf algebra 
and gave an explicit formula for 
both the antipode $\sigma$ and its inverse $\sigma^{-1}$, 
provided that $\A$ is a category of quiver representations. 
It can be shown that Xiao's formulas hold more generally 
provided that $\A \subset \R$ satisfies \eqref{sub}. 

Athough the formula for $\sigma$ is no longer well-defined  
in the case where $\A$ does not satisfy \eqref{sub}, 
there is a more general condition 
that ensures the formula for $\sigma^{-1}$ is still defined. 
Recall from \cite{Cram} that an {\em anti-equivalence} between two objects $A,B \in \A$ 
is a pair of strict filtrations 
\[
0=L_{n+1} \subsetneq L_{n} \subsetneq \dots L_{1} \subsetneq L_{0} = A, 
\qquad
0=L'_{n+1} \subsetneq L'_{n} \subsetneq \dots L'_{1} \subsetneq L'_{0} = B
\]
%
such that $L'_{i}/L'_{i+1}\cong L_{n-i}/L_{n-i+1}$
for all $i$. 
Two objects $A$ and $B$ in $\A$ are called {\em anti-equivalent}
if there exists at least one anti-equivalence between them.

It follows from results in \cite{BS} and \cite{Cram} 
that Xiao's formula for the map $\sigma^{-1}$ is still well-defined 
provided that 
the following condition holds
for every pair of objects $A,B\in \A$: 
\begin{equation}
\label{sigma1}
\text{\em 
There are finitely many anti-equivalences (if any) 
between $A$ and $B$.}
\end{equation}
Since the map $\sigma^{-1}$ generally 
takes values in a certain completion 
of $\widetilde{\H_v}(\R)$, an extra condition is needed to ensure 
that corresponding relations in the bialgebra, such as 
%
\[m\circ (\sigma^{-1}\otimes \id) \circ \Delta^{\rm{op}} = i \circ \epsilon, \]
are still well-defined. 
%
The required condition can be stated as follows: 
\begin{gather}
\label{sigma2}
\text{\em 
Given any $A,B\in \A$,
there are finitely many pairs $(A',B')$ of anti-equivalent }\\
\text{\em objects  such that $A'\hookrightarrow A$ and $B\twoheadrightarrow B'$.} \nonumber
\end{gather}
Up to minor modifications, the formulas for (and corresponding properties of) $\sigma$ and $\sigma^{-1}$
continue to hold, respectively, in $\widetilde{H_v}(\R)$, whenever they are defined in $\widetilde{H_v}(\A)$. 

If the conditions 
\eqref{sigma1} and \eqref{sigma2}
are both satisfied by the category $\A\subset \R$, 
then the {\em Drinfeld double} of $\H= \widetilde{\H_v}(\R)$ is the vector space 
$\H\otimes \H$ equipped with the multiplication $\circ$ 
uniquely determined by the following conditions:
\begin{itemize}
\item [(D1)]
The maps 
\[
\H \longrightarrow \H \otimes_{\CC} \H,
 \quad
 a \longmapsto a\otimes 1
\]
and
\[
\H \longrightarrow \H \otimes_{\CC} \H,
 \quad
 a \longmapsto 1 \otimes a
\]
are injective homomorphisms of $\CC$-algebras;
\smallskip

\item [(D2)]
For all elements $a,b \in \H$, one has
\[
 (a\otimes 1)\circ(1\otimes b)=a \otimes b;
\]
\item [(D3)]
For all elements $a,b \in \H$, one has
\begin{equation*}
(1\otimes b) \circ (a\otimes 1) 
=  \sum 
  (b_{(1)},a_{(3)})_H 
  (\sigma^{-1}(b_{(3)}),a_{(1)})_H 
  \cdot a_{(2)} \otimes b_{(2)}
\end{equation*}
where 
$\Delta^2(a)=\sum a_{(1)} \otimes a_{(2)} \otimes a_{(3)}$
and  
$\Delta^2(b)=\sum b_{(1)} \otimes b_{(2)}\otimes b_{(3)}$.
\end{itemize}
The last identity is equivalent to  
%
\begin{equation}
\label{eq:Drinf}
  \sum (a_{(2)},b_{(1)})_H \cdot
 a_{(1)} \otimes b_{(2)}
 =\sum (a_{(1)},b_{(2)})_H \cdot
    (1 \otimes b_{(1)}) \circ (a_{(2)} \otimes 1)
\end{equation}
for all $a,b\in \H$, 
where 
$\Delta(a)=\sum a_{(1)} \otimes a_{(2)}$ and 
$\Delta(b)=\sum b_{(1)} \otimes b_{(2)}$.

An argument similar to the proof given in \cite[Lemma 3.2.2]{Jo} 
shows that the multiplication $\circ$ is associative. 
%
%
%

\section{Hall algebras of complexes}
\label{sec:cplx}
Assume that $\R$ is an abelian category for which the triple $(\R, \P, \A)$ satisfies  axioms (a)-(e) of Section \ref{ass}. We now introduce certain categories of complexes over $\R$ and define corresponding Hall algebras and their localizations.
The associativity of multiplication in the localized Hall algebras will be established later in Section \ref{sec:dbl}.

\subsection{Categories of complexes}
Define a $\Zt$-graded chain complex in $\R$ to be a diagram 
\[ \cplx{M_1}{d_1}{d_0}{M_0}\] 
such that $d_i \circ d_{i+1} = 0$ for all $i\in \Zt$.  


A morphism $s_\blob \colon M_\blob\to \tilde{M}_\blob$ consists of a diagram
\[\begin{tikzcd}
M_1\ar[d, "s_1" '] \ar[r, "d_1", shift left]  & M_0 \ar[l, "d_0", shift left]  \ar[d, "s_0"] \\
\tilde{M}_1\ar[r,"\tilde{d}_1", shift left] & \tilde{M}_0 \ar[l, "\tilde{d}_0", shift left]
\end{tikzcd}\]
with $s_{i+1}\circ d_{i} = \tilde{d}_i\circ s_{i}$.

Let $\C_\Zt(\R)$ denote the category of all 
$\Zt$-graded 
chain complexes in $\R$ 
with morphisms defined above. 
Two morphisms $s_\blob, t_\blob\colon M_\blob\to \tilde{M}_\blob$ are {\em homotopic}  
if there are morphisms $h_i\colon M_i\to \tilde{M}_{i+1}$ such that
\[ t_i-s_i=\tilde{d}_{i+1}\circ h_{i}+h_{i+1}\circ d_{i}.\]
We write $\Ho_\Zt(\R)$ for the category obtained from $\C_\Zt(\R)$ by identifying homotopic morphisms.

The {\em shift functor} defines an involution 
\[\C_\Zt(\R) \qquad \stackrel{\dag}{\longleftrightarrow} \qquad \C_\Zt(\R)\]  which shifts the grading and changes the sign of the differential
\[ 
\begin{tikzcd}[cramped, column sep=0.75cm] M_1 \ar[r, "d_1", shift left] & M_0 \ar[l, "d_0", shift left] 
\end{tikzcd} 
\qquad
\stackrel{\dag}{\longleftrightarrow}  
\qquad 
\begin{tikzcd}[cramped, column sep=0.75cm] M_0 \ar[r, "-d_0", shift left] & M_1 \ar[l, "-d_1", shift left] 
\end{tikzcd} . \] 
The image of $M_\blob$ under the shift functor will be denoted by $M^*_\blob$. Every complex $M_\blob\in \C_\Zt(\R)$ defines a class
$\cl{M_\blob}:=\cl{M_0}-\cl{M_1}\in K(\R)$ in the Grothendieck group of $\R$. 

We are mostly concerned with the 
full subcategories  
\[ \C_\Zt(\P),\ \C_\Zt(\A) \subset \C_\Zt(\R) \quad  
\text{and} \quad
\Ho_\Zt(\P) \subset \Ho_\Zt(\R)\]
consisting of complexes of objects in  
$\P$ and $\A$, respectively. 
%



\subsection{Root category}
Let $\D^b(\R)$ denote the ($\ZZ$-graded) bounded derived category of $\R$, with its shift functor $[1]$. 
Let $\Ro(\R)=\D^b(\R)/[2]$ be the {\em orbit category}, also known as the {\em root category} of $\R$. This has the same objects as $\D^b(\R)$, but the morphisms are given by
\[\Hom_{\Ro(\R)} (X,Y):=\bigoplus_{i\in \ZZ} \Hom_{\D^b(\R)}(X,Y[2i]).\]
 Since $\R$ is an abelian category of finite global dimension ($\leq 1$) 
with  enough projectives, the category $\D^b(\R)$ is equivalent to the ($\ZZ$-graded) bounded homotopy category 
$\Ho^b(\P)$.
Thus we can equally well define $\Ro(\R)$ as the orbit category of $\Ho^b(\P)$.

 \begin{lem}[\cite{Br}]
 \label{root}
 There is a fully faithful functor
\[D\colon \Ro(\R)\lra \Ho_{\Zt}(\P)\]
sending a $\ZZ$-graded complex of projectives $(P_i)_{i\in\ZZ}$ to the $\Zt$-graded complex
\[
\cplx{\bigoplus_{i\in \ZZ} P_{2i+1}}{0}{0}{\bigoplus_{i\in\ZZ} P_{2i}}.
\]
\end{lem}

\subsection{Decompositions} \label{sec-decom}
From now on, we omit $\Zt$ in the notations for categories and write
\begin{align*} 
\C(\P)&=\C_\Zt(\P),\ &\C(\A)&=\C_\Zt(\A), \ &\C(\R)&=\C_\Zt(\R),\\  \Ho(\P)&= 
\Ho_\Zt(\P), \ &\Ho(\R)&= \Ho_\Zt(\R). && \end{align*}
The homology of a complex $M_\blob \in \C(\R)$ will be denoted 
\[H_\blob(M_\blob)\,  := \ 
(\begin{tikzcd}[cramped, column sep=scriptsize] H_1(M_\blob) \ar[r, "0", shift left] & H_0(M_\blob) \ar[l, "0", shift left] 
\end{tikzcd}) \ 
 \in\, \C(\R). \] 
To each morphism $f:P\to Q$ in the category $\P$, we associate the following complexes
\begin{equation}\label{eq:comp}
C_{f} :=
(\begin{tikzcd}[cramped, column sep=scriptsize] P \ar[r, "f", shift left] & Q \ar[l, "0", shift left] 
\end{tikzcd}), 
\qquad 
C^\dag_{f} :=
(\begin{tikzcd}[cramped, column sep=scriptsize] Q \ar[r, "0", shift left] & P \ar[l, "-f", shift left] 
\end{tikzcd})
\end{equation}
in $\C(\P)$. 

\begin{lem}
\label{decomp}
Every complex of projectives  
$M_\blob \in \C(\P)$ can be decomposed  uniquely,  up to isomorphism, 
as a direct sum of complexes of the form  
\[M_\blob = C_{f} \oplus C_{g}^\dag\]
for some injective morphisms $f, g$ in $\P$ 
such that 
$H_0(M_\blob) = \cok(f)$
and 
$H_1(M_\blob) = \cok(g)$.
\end{lem}
\begin{proof}
Consider the short exact sequences
\[0\lra \ker(d_1)\lRa{i} M_1\lRa{p} \im(d_1)\lra 0,\]\[ 0\lra \ker(d_0)\lRa{j} M_0\lRa{q} \im(d_0)\lra 0.\]
Since the the category $\R$  is hereditary by assumption, 
all the objects appearing in these sequences are projective. 
Thus the sequences split, and we can find morphisms
\[r\colon M_1\to \ker(d_1), \quad k\colon \im(d_0)\to M_0, \quad
l\colon \im(d_1)\to M_1, \quad s\colon M_0\to \ker(d_0)
\]
such that $r\circ i=\id$, $q\circ k=\id$,  $p\circ l=\id$, and $s\circ j=\id$. 
This yields the following split exact sequence of morphisms of complexes
\[
\begin{tikzcd}
\ker(d_1)\ar[d, "i" '] \ar[r, "0", shift left] & \ar[l, "m", shift left]	 \im(d_0) \ar[d, "k"]
\\
M_1 \ar[d, "p" ']  \ar[r, "d_1", shift left]	& \ar[l, "d_0", shift left]  M_0 \ar[d, "s"]
\\
\im(d_1) \ar[r, "m'", shift left] & \ar[l, "0", shift left]	 \ker(d_0)
\end{tikzcd}
\]
where $m, m'$ denote the obvious inclusions. (Note that $d_0=i\circ m\circ q$ and $d_1=j\circ m'\circ p$.)
The desired decomposition of $M_\blob$ is thus given by 
setting: $f=m\circ q$,\, $g=-m' \circ l$. 

Now suppose there is an isomorphism $M_\blob \cong C_{f'}\oplus C_{g'}^\dag$ for some other pair $f',g'$ of injective morphisms in $\P$. 
Then one can easily define corresponding isomorphisms of complexes $C_{f'} \cong C_f$ and $C_{g'} \cong C_g$, 
showing uniqueness. 
\end{proof}

Given $M_\blob\in \C(\P)$, 
it will be convenient to write the decomposition in Lemma \ref{decomp} as 
\begin{equation*}
M_\blob = M_\blob^+ \oplus M_\blob^-
\end{equation*} 
where $M_\blob^+ = C_f$ and $M_\blob^- = C_g^*$. 
Let the sign map \[\varepsilon: \Zt \to \{+,-\}\] be defined by 
$\varepsilon(0)=+$, $\varepsilon(1)=-$.

\begin{lem}\label{compare}
Let $M_\blob, N_\blob \in \C(\P)$. Then there is an isomorphism 
\begin{align*}
\Hom_{\C(\R)}(M_\blob, N_\blob) 
\, \cong \  
 \Hom_{\C(\R)}(H_\blob(M_\blob), H_\blob(N_\blob))
  \oplus
  \Big \{ 
  \bigoplus_{i,j\in\Zt} 
\Hom_\R(M_i^{\varepsilon(j)}, N_{j+1}^{\varepsilon(i)} )  
\Big \}
\end{align*} 
%
of $\k$-vector spaces. 
\end{lem}

\begin{proof}

First suppose that 
 $P\xrightarrow{\,f} Q$, $P' \xrightarrow{\,g} Q'$ 
is a pair of injective morphisms  in $\P$. 
We note that there is a short exact sequence 
\begin{equation}\label{hom1}
 0 \lra \Hom_\R(Q,P') 
\lra \Hom_{\C(\R)}(C_f, C_g)
\lra \Hom_{\R}(X,Y) \lra 0 
\end{equation}
for $X=\cok f$, $Y=\cok g$.  
One may also check directly that 
\begin{equation}\label{hom2}
 \Hom_{\C(\R)}(C_f, C_g^\dag) \cong \Hom_\R(P,Q').
 \end{equation}

The decomposition of $\Hom_{\C(\R)}(M_\blob, N_\blob)$ 
now follows easily
 by applying Lemma \ref{decomp} to both 
 $M_\blob$ 
and  $N_\blob$, respectively, and by 
 using the involutive shift functor $*$ together with 
 \eqref{hom1} and \eqref{hom2}. 
\end{proof}

\subsection{Acyclic complexes} \label{ss:acyclic}

Given a projective object $P\in \P$, there are associated acyclic complexes  
\begin{equation}\label{def:complex2}
K_P := 
(\begin{tikzcd}[cramped, column sep=small] P \ar[r, "\text{id}", shift left] & P \ar[l, "0", shift left] 
\end{tikzcd}), \ \quad \
K^\dag_P := 
(\begin{tikzcd}[cramped, column sep=small] P \ar[r, "0", shift left] & P \ar[l, "-\text{id}", shift left] 
\end{tikzcd}).
\end{equation}
Notice that $M_\blob \in \C(\P)$ is acyclic precisely if $M_\blob \cong 0$ in $\Ho(\P)$.  

\begin{lem}\label{acyclic}
If  $M_\blob\in \C(\P)$  is an acyclic complexes of projectives,  
then there exist objects $P,Q\in \P$, unique up to isomorphism, such that 
$M_\blob \cong K_P\oplus K_Q^\dag$. 
\end{lem}
\begin{proof}
If $M_\blob$ is acyclic, then by Lemma \ref{decomp}
we have  
$M_\blob = C_{f} \oplus C_{g}^\dag$ 
for some isomorphisms 
$f:P\cong P'$ and $g:Q\cong Q'$ of projectives. 
It follows that  $M_\blob\cong K_P \oplus K_Q^\dag$. 
Since the complexes
 $K_P$ and $K_Q^\dag$ are unique up to isomorphism, 
the objects $P,Q$ are unique up to isomorphism as well. 
\end{proof}

\begin{lem}\label{Cohn}
Suppose that 
 $P\xrightarrow{\,f} Q$ and $P' \xrightarrow{\,f'} Q'$ 
are injective morphisms  in $\P$.  
Then 
$\cok f\cong \cok f'$ in $\R$, 
if and only if 
there is an isomorphism 
\begin{equation*} 
C_f\oplus K_{L'} \, \cong \,  K_{L } \oplus C_{f'} 
\end{equation*}
of complexes in $\C(\P)$, 
for some objects $L,L' \in \P$.
\end{lem}

\begin{proof}
This is a reformulation of Schanuel's lemma.  We refer to \cite[Theorem 0.5.3]{Cohn} for the proof. 
\end{proof}

\begin{prop}\label{exchange}
Suppose $M_\blob, M_\blob' \in \C(\P)$.  Then 
%
%
there exists an isomorphism 
\[ M_\blob \oplus K_\blob' \, \cong \,  K_\blob \oplus M_\blob' \]
for some acyclic complexes $K_\blob, K_\blob' \in \C(\P)$,  
if and only if 
$H_\blob(M_\blob)\cong H_\blob(M_\blob')$ in $\C(\R)$. 
\end{prop}

\begin{proof}
This follows directly from Lemmas \ref{decomp}
and \ref{Cohn}, and by applying $\dag$ to the latter.  
\end{proof}

\subsection{Extensions of complexes} 
Given any morphism $s_\blob: M_\blob \to N_\blob$ of complexes in $\C(\P)$, 
we can form a corresponding exact sequence 
\[ 0 \lra N_\blob^\dag \lra \mathrm{Cone}(s_\blob) \lra M_\blob \lra 0\] 
of complexes in $\C(\P)$, 
where the middle term is defined by 
\[
\mathrm{Cone}(s_\blob) = 
(\begin{tikzcd}[cramped, column sep=scriptsize] N_0 \oplus M_1 \ar[r, "d_1", shift left] & N_1 \oplus M_0 \ar[l, "d_0", shift left] 
\end{tikzcd}) 
\]
with 
\[
d_0 := \begin{bmatrix} -d_1^N & s_0 \\ 0 & d^M_0 \end{bmatrix}, \qquad \quad
d_1 := \begin{bmatrix} -d_0^N & s_1 \\ 0 & d^M_1 \end{bmatrix}. 
\] 
This leads to the following result. 

\begin{lem}[\cite{Br}]\label{ext_hom} 
Let $M_\blob, N_\blob \in \C(\P)$.  The mapping $s_\blob \mapsto \mathrm{Cone}(s_\blob)$ 
defines an isomorphism
\[\Hom_{\Ho(\R)}(M_\blob, N_\blob^\dag) \, \cong \ 
\Ext^1_{\C(\R)}(M_\blob, N_\blob).\]
\end{lem}

We also have the following.

\begin{lem}\label{ext_hom2}
Suppose 
$P\xrightarrow{f}Q$, 
$P'\xrightarrow{g}Q'$  
are injective morphisms in the category $\P$. 
Let $X,Y \in \R$ denote the cokernels: 
$X= \cok f$, $Y=\cok g$.  
Then the following hold. 
\begin{enumerate}[(i)]
\item 
$\Hom_{\Ho(\R)}(C_{f}, C_{g}) \cong 
\Hom_{\R}(X, Y)$; \smallskip

\item 
$\Hom_{\Ho(\R)}(C_{f}, C_{g}^\dag) \cong 
\Ext^1_{\R}(X, Y)$.
\end{enumerate}
\end{lem}
\begin{proof}
The category $\R$ can be identified as a full subcategory of $\D^b(\R)$ by considering any object in $\R$ as a complex concentrated in degree 0. It follows that 
\[\Hom_\R(X,Y) \cong \Hom_{\D^b(\R)}(X,Y) \cong \Hom_{\Ro(\R)}(X,Y).\]
The objects $X,Y$ have the projective resolutions 
\[ 0\to P\xrightarrow{\, f}Q \twoheadrightarrow X \to 0,  \qquad
0\to P'\xrightarrow{\, g}Q' \twoheadrightarrow Y \to 0.\]
So the complexes $C_{f}, C_{g}$ are quasi-isomorphic to $X,Y$ respectively, 
and isomorphisms (i), (ii) thus follow by Lemmas \ref{root}.  
\end{proof}

Note that the isomorphism in Lemma \ref{ext_hom2} (i) may be given explicitly by $s_\blob \mapsto s$, where 
$s: A\to B$ is the unique morphism making the diagram
\begin{equation}\label{homology}
\begin{tikzcd}
0 \rar & P \rar \dar["s_1"] & Q \rar \dar["s_0"] & X  \dar["s"] \rar & 0 \\
0 \rar & P' \rar  & Q' \rar  & Y  \rar & 0
\end{tikzcd}
\end{equation}
commutative.

\subsection{Hall algebras of complexes}

We denote by $\C_\fin(\P)$ the full subcategory of $\C(\P)$ consisting 
of all complexes with finite homology, i.e.~complexes $M_\blob$ such that 
\[ H_\blob(M_\blob)\,  = \ 
(\begin{tikzcd}[cramped, column sep=scriptsize] H_1(M_\blob) \ar[r, "0", shift left] & H_0(M_\blob) \ar[l, "0", shift left] 
\end{tikzcd}) \ 
 \in\, \C(\A). \] 
The following result will be crucial for our definition of the Hall algebra of $\C_\fin(\P)$.

\begin{lem}\label{homotopy}
The set $\Ext^1_{\C(\R)}(M_\blob,N_\blob)$ is finite for all $M_\blob, N_\blob \in \C_\fin(\P)$. 
\end{lem}

\begin{proof}
This follows by using the involution $\dag$ together with Lemma \ref{ext_hom2} and combining 
with Lemmas \ref{decomp} and \ref{ext_hom}. 
\end{proof}

Since the category $\C_\fin(\P)$ is not necessarily Hom-finite, we must consider 
a generalization of 
the coefficients appearing in the definition \eqref{defprod} 
of the Hall product. 
First define a bilinear map, 
$\mu: \C(\P)\times \C(\P) \to \mathbb{Q}$, 
given by 
\smallskip
\begin{align}\label{mu}
\mu(M_\blob,N_\blob) \ :=&\ 
\< \hat M_0^+, \hat N_1^+ \> +  \< \hat M_1^+, \hat N_1^- \> + 
\< \hat M_0^-, \hat N_0^+ \> + \< \hat M_1^-, \hat N_0^- \> 
\\[.25cm]
=&  \, \ 
 \sum_{i,\, j} 
\big\< \hat M_i^{\varepsilon(j)}, \hat N_{j+1}^{\varepsilon(i)} \big\> \nonumber
\end{align}
where 
$\< \ ,\ \>: K(\R) \times K(\R) \to \mathbb{Q}$ 
denotes the generalized Euler form (Section \ref{ss:Euler}). 

Then for any
$M_\blob, N_\blob, P_\blob \in\C_\fin(\P)$,  
we define 
\begin{equation}\label{h}
\h(M_\blob,N_\blob) \, := \, 
q^{\mu(M_\blob,N_\blob)}\, 
|\Hom_{\C(\R)}(H_\blob(M_\blob), H_\blob(N_\blob))| 
\end{equation}
where $q$ is the cardinality of $\k$, 
and 
we also write 
\begin{equation*}
\e(M_\blob,N_\blob)_{P_\blob} \, := \, 
|\Ext^1_{\C(\R)}(M_\blob,N_\blob)_{P_\blob}| 
\end{equation*}
which is well-defined
by Lemma \ref{homotopy}.

In the remainder, let us write  
 $\X = \Iso(\C_\fin(\P))$ 
 for the set of isomorphism classes in $\C_\fin(\P)$. 
\begin{definition}\label{def:Hall}
The Hall algebra 
$\H(\C_\fin(\P))$ is defined 
to be the $\CC$-vector space with basis elements $\gen{M_\blob}$  
 indexed by isoclasses  
$M_\blob \in \X$, 
and with multiplication defined by 
\begin{equation*}
\gen{M_\blob} \ostar \gen{N_\blob}\ := 
v^{\<\hat M_0,\hat N_0\> + \<\hat M_1,\hat N_1\>} 
\sum_{P_\blob\in\X} 
\frac{\, \e(M_\blob,N_\blob)_{P_\blob}\, }{\h(M_\blob,N_\blob)} 
\, \gen{P_\blob},\end{equation*}
for all $M_\blob, N_\blob \in \C_\fin(\P)$. 
\end{definition}

\begin{remark}\label{Bridg}
Suppose there are complexes $M_\blob, N_\blob\in \C_\fin(\P)$  
 such that $\Hom_{\C(\R)}(M_\blob,N_\blob)$ is a finite set. 
Then it can be checked using Lemma \ref{compare} and the definition of Euler form that 
\[
 |\Hom_{\C(\R)}(M_\blob,N_\blob)|
= \h(M_\blob, N_\blob). \]
The above definition thus generalizes the 
(twisted) Hall algebras of complexes of projectives defined by Bridgeland in \cite{Br}. 
\end{remark}

\subsection{Localization} \label{ss:local}

As before, let us write 
$\cl{M_\blob} = \cl{M_0} - \cl{M_1} \in K(\P)$, 
for each $M_\blob \in \C(\P)$. 
The following result shows that the acyclic complexes $K_P$ introduced in Section \ref{ss:acyclic} define elements of $\H(\C_\fin(\P))$ with particularly simple properties.

\begin{lem}
\label{easy}
For any projective object $P\in \P$ and any complex $M_\blob\in\C_\fin(\P)$ the following identities hold in $\H(\C_\fin(\P)){\,:}$
\begin{align}\label{e1}
\gen{K_P}\ostar \gen{M_\blob}&= v^{\<\cl{P},\cl{M_\blob}\>}\cdot\gen{K_P\oplus M_\blob},\\
\label{e2}
\gen{M_\blob}\ostar  \gen{K_P}&= v^{-\<\cl{M_\blob},\cl{P}\>}\cdot \gen{ K_P\oplus M_\blob}.\end{align}
\end{lem}

\begin{proof} 
It is easy to check directly from \eqref{mu} that
\[\mu(K_P,M_\blob)=
\<\hat P,\hat M^+_1\> + \<\hat P,\hat M^-_1\> 
, \quad 
\mu(M_\blob,K_P) = \<\hat M_0^+,\hat P\> + \<\hat M_0^-,\hat P\> \]
so that 
\[\h(K_P, M_\blob)= q^{\<\hat P,\hat  M_1\>},\quad 
\h(M_\blob,K_P) = q^{\<\hat M_0,\hat P\>}.
\]
The complexes $K_P$ are homotopy equivalent to the zero complex, 
so Lemma \ref{ext_hom} shows that the extension group in the definition of the Hall product vanishes. 
Taking into account Definition \ref{def:Hall}  gives the result. 
\end{proof}

\begin{lem}
\label{cor}
For any projective object $P\in \P$ and any complex $M_\blob\in\C_\fin(\P)$ the following identities are true in $\H(\C_\fin(\P)){\,:}$
\begin{align}\label{eq1} \gen{K_P}\ostar \gen{M_\blob}=v^{(\cl{P},\cl{M_\blob})} \, \gen{M_\blob}\ostar  \gen{K_P},\\
\label{eq2}
\gen{K_P^\dag}\ostar \gen{M_\blob}=v^{-(\cl{P},\cl {M_\blob})} \, \gen{M_\blob}\ostar \gen{K_P^\dag}.\end{align}
\end{lem}

\begin{proof}
Equation \eqref{eq1} is immediate from Lemma \ref{easy}. Equation \eqref{eq2} follows  by applying the involution $\dag$.
\end{proof}

In particular, since $\cl{K_{P}}=0\in K(\R)$, we have  for $P, Q \in \P$,
\begin{equation}\label{eq3}
\gen{K_P} \ostar  \gen{K_Q} = \gen{K_{P}\oplus K_{Q}}, \quad \gen{K_P}\ostar  \gen{K_Q^\dag}=\gen{K_P\oplus K_Q^\dag},\end{equation}
\begin{equation}\label{eq4}
\big [\gen{K_P},\gen{K_Q} \big ]=\big [\gen{K_P},\gen{K_Q^\dag} \big ]=\big [\gen{K_P^\dag},\gen{K_Q^\dag}\big ]=0,  \end{equation} where $[x,y] := x\ostar y - y\ostar x$.
Note that  any element of the form $\gen{K_{P}}\ostar \gen{K_P^\dag}$ is central.

Let us write $\C_0 \subset \C_\fin(\P)$ to denote the full subcategory of all acyclic complexes of projectives. 
It then follows from \eqref{eq3} and \eqref{eq4} that the subspace $\H(\C_0) \subset \H(\C_\fin(\P))$
spanned by the isoclasses of objects in $\C_0$
is closed under the multiplication $\ostar$ 
and has the structure of a commutative associative algebra.
%
%

The following result is also clear. 

\begin{lem}
The left and right actions of $\H(\C_0)$ on $\H(\C_\fin(\P))$ given 
by restricting multiplication 
make the Hall algebra $\H(\C_\fin(\P))$ into an 
$\H(\C_0)$-bimodule. 
\end{lem}

Notice that the basis   
$\Z := \{\gen{M_\blob} \in \H(\C_0)\}$
is a multiplicative subset in $\H(\C_0)$. 
Let us write 
$\DH_0(\R) = \H(\C_0)_\Z$ 
to denote the localization of 
$\H(\C_0)$ at $\Z$. 
More explicitly, we have
\[\DH_0(\R) = \H(\C_0)\big [ \gen{M_\blob}^{-1} :
M_\blob \in \C_0 \big ].\]

The assignment $P\mapsto K_P$ extends to a  group homomorphism
\[K\colon K(\R)\lra \DH_0(\R)^\times.\]
This map is given explicitly by writing an element $\alpha\in K(\R)$ in the form  $\alpha=\cl{P}-\cl{Q}$ for objects $P,Q\in \P$ and then setting $K_\alpha := K_{\cl{P}} \ostar K_{\cl{Q}}^{-1}$. 
Composing with the involution $\dag$ gives another map
\[K^\dag\colon K(\R)\lra \DH_0(\R)^\times.\]
Taking these maps together an extending linearly defines a $\CC$-linear map from the group algebra
\[\CC[K(\R)\times K(\R)]\ \isoto\ \DH_0(\R),\]
which is an isomorphism by Lemma \ref{acyclic}. 
It follows that the set 
$\{K_{\alpha} \ostar K_{\beta}^\dag \mid   \alpha, \beta \in K(\R) \}$  
gives a $\CC$-basis of  $\DH_0(\R)$

\begin{definition}
The {\em localized Hall algebra}, $\DH(\R)$, 
is the right $\DH_0(\R)$-module
 obtained from $\H(\C_\fin(\P))$  
by extending scalars, 
\[\DH(\R) 
:= \H(\C_\fin(\P))\otimes_{\H(\C_0)} \DH_0(\R). \]
That is, $\DH(\R)$ is  the localization 
$\H(\C_\fin(\P))_\Z$.
We also consider $\DH(\R)$ as a $\DH_0(\R)$-bimodule by setting 
\begin{equation}\label{commute}
 (K_{\alpha}\ostar K_{\beta}^\dag) \ostar \gen{M_\blob} := 
v^{(\alpha - \beta, \cl{M_\blob})} \cdot 
\gen{M_\blob} \ostar (K_{\alpha}\ostar  K_{\beta}^\dag)  
\end{equation}
for all $M_\blob \in \C_\fin(\P)$ and $\alpha, \beta \in K(\R)$.  
We thus have a well-defined binary operation 
\[ - \ostar -: \DH(\R) \times \DH(\R) \to \DH(\R)\] 
which agrees with the  map induced from  
the multiplication 
in Definition \ref{def:Hall} by restricting along 
the canonical map $\H(\C_\fin(\P)) \to  \DH(\R)$. 
\end{definition}

Given a complex $M_\blob\in \C_\fin(\P)$, define a corresponding element $E_{M_\blob}$ in $\DH(\R)$ given by
\[E_{M_\blob} :=  v^{\<\cl{M_1^+} -\cl{M_0^-}, \cl{M_\blob}\>}\, 
K_{-\cl{M_1^+}} \ostar  K_{-\cl{M_0^-}}^\dag \ostar \gen{M_\blob}.\]
Then we claim that 
\begin{equation}\label{EK}
E_{M_\blob\oplus K_\blob}= E_{M_\blob}
\end{equation}
for any acyclic complex of projectives 
$K_\blob \in \C_0$. 
Indeed, suppose that $K_\blob = K_P \oplus K_Q^\dag$ for some $P,Q\in \P$.  
Then clearly 
$M_\blob\cl{\oplus}K_\blob = \cl{M_\blob}$, and 
it follows by Lemma \ref{easy} that 
\begin{align*}
E_{M_\blob\oplus K_\blob}
& =  
v^{\<M_1^+\cl{\oplus}P-M_0^-\cl{\oplus}Q,\, M_\blob\cl{\oplus}K_\blob\>}
\cdot 
K_{-M_1^+\cl{\oplus}P} \ostar  K_{-M_0^-\cl{\oplus}Q}^\dag \ostar \gen{K_P\oplus K_Q^\dag \oplus M_\blob} \\[1em]
& = v^{\<\cl{M_1^+}-\cl{M_0^-},\, \cl{M_\blob}\>}
\cdot 
K_{-\cl{M_1^+}} \ostar  K_{-\cl{M_0^-}}^\dag \ostar \gen{M_\blob}, 
\end{align*}
so we get the same element $E_{M_\blob}$.

We note that a minimal projective resolution of $A \in \A$ need not be unique because the category $\C(\P)$ is not Krull--Schmidt in general. However, it will be convenient to fix a (not necessarily minimal) resolution for each object $A \in A$.

\begin{definition} \label{def-eab}
(i) For each object $A \in \A$, fix a projective resolution
\begin{equation} \label{eqn-resolpqa}
0\lra P_A \stackrel{f_A}{\lra}Q_A \lra A\lra 0 \end{equation}
and the complex $C_{A}$ is defined to be $C_{f_A}\in \C(\P)$.

(ii) Given objects $A, B \in \A$, 
write $E_{A,B}$ to denote the element $E_{C_A \oplus C_B^\dag}$ in $\DH(\R)$. 
\end{definition}

The next lemma shows that the definition of $E_{A,B}$ 
 is independent of the choice of resolutions defining $C_A$ and $C_B$. 
\begin{lem}\label{lem-E}
Suppose $A,B\in \A$, 
and let $M_\blob\in \C_\fin(\P)$ be any complex such that 
$A \cong H_0(M_\blob)$ and $B\cong H_1(M_\blob)$. 
Then $E_{M_\blob} = E_{A,B}$. 
\end{lem}

\begin{proof}
Let $M_\blob$ be such a complex. 
By Proposition \ref{exchange} there exist 
acyclic complexes $K_\blob, K_\blob'$ in $\C_0(\P)$ such that 
$\gen{M_\blob \oplus K_\blob} \cong \gen{C_A\oplus C_B^\dag \oplus K_\blob'}$, 
and the result follows from \eqref{EK}.
\end{proof}

The following result provides an explicit basis for the localized Hall algebra. 

\begin{prop}\label{free}
The algebra $\DH(\R)$ is free as a right
$\DH_0(\R)$-module, with basis consisting of elements
$E_{A,B}$ 
indexed by all pairs of objects 
$A,B \in \Iso(\A)$. 
\end{prop}

\begin{proof}
Suppose $M_\blob \in \X$,  and set $A=H_0(M_\blob)$, $B=H_1(M_\blob)$.  
Then $E_{M_\blob} = E_{A,B}$ by Lemma \ref{lem-E}, 
and one may check using \eqref{commute} that 
\[ \gen{M_\blob} = 
v^{\<\cl{Q} - \cl{P}, \cl{M_\blob}\>}\, 
 E_{A,B} \ostar \gen{K_P \oplus K_Q^\dag},
 \]
for 
$P=M_1^+$ and $Q=M_0^-$. 
This shows that the elements $E_{A,B}$ 
span $\DH(\R)$ as a $\DH_0(\R)$-module.

It remains to check that the elements $E_{A,B}$ 
are $\DH_0(\R)$--linearly independent. 
Notice that the Hall algebra $\H(\C_\fin(\P))$ 
is naturally graded as a $\CC$-vector space by the set $\Iso(\A)\times \Iso(\A)$: 
\begin{equation*}
\H(\C_\fin(\P)) = \bigoplus_{(A,B) \in \Iso(\A)^2} \H_{(A,B)}, \qquad
\H_{(A, B)}  :=  
\bigoplus_{H_0(M_\blob)\simeq A,\, H_1(M_\blob)\simeq B} 
\CC\gen{M_\blob}.
\end{equation*}
Since the action of $\H(\C_0)$ on $\H(\C_\fin(\P))$ is $\Iso(\A)^2$-homogeneous, it follows that 
$\DH(\R)$ also has an $\Iso(\A)^2$-grading.  It is thus clear that the elements $E_{A,B}$ span distinct graded components of $\DH(\R)$. 
To see that each component is a  free $\DH_0(\R)$-module of rank one, it remains to check that for each $y\in \DH(\R)$, we have $E_{A,B}\ostar y=0$ implies $y=0$. 

Let us write $M_\blob = C_A \oplus C_B^\dag$. 
Then it will suffice to show that for any $x\in \H(\C_0)$, the element 
\[\gen{M_\blob} \ostar x \in \H(\C_\fin(\P))\]
is a $\Z$-torsion element only if $x=0$. 
Suppose that 
\[x = c_1 z_1+ \cdots + c_n z_n \]
for some constants $c_1, \dots, c_n\in \CC$ and distinct elements $z_1, \dots, z_n\in \Z$. 
One may check using Lemma \ref{acyclic} and condition (c) in Section \ref{ass}, that for any $z\in \Z$ the elements $z_1\ostar z, \dots, z_n\ostar z$ are also distinct. 
Next suppose that $\gen{M_\blob}\ostar (x\ostar z) = 0$ in $\H(\C_\fin(\P))$. 
This gives an equation
\[
c_1\cdot \gen{M_\blob}\ostar z_1\ostar z
+ \cdots + 
c_n \cdot \gen{M_\blob}\ostar z_n\ostar z=0.
\]
One may again use condition (c) together with Lemmas \ref{decomp} and \ref{acyclic}
to check that the terms appearing in this dependence relation are unit multiples of distinct basis elements in $\X$.  So the relation must be trivial: $c_1 = \dots = c_n =0$, which gives $x=0$. 
This completes the proof. 
\end{proof}

\section{Associativity via the Drinfeld double}\label{sec:dbl}
In this section 
we prove that $\DH(\R)$ is the Drinfeld double of the bialgebra $\widetilde{\H_v}(\R)$ under suitable finiteness conditions. 
As a corollary, we show that $\DH(\R)$ is an associative 
algebra with respect to the multiplication described in the previous section.

\subsection{Multiplication formulas }
Suppose $A,B\in \A$ and recall the element $E_{A,B}$ in $\DH(\R)$ defined in Definition \ref{def-eab}. 
Notice that the image under the involution $\dag$ is given by 
$E_{A,B}^\dag = E_{B,A}$. 
Let us write 
\[E_A := E_{A,0}, \qquad 
F_B := E_{0,B}\] 
so that  $F_A = E_A^\dag$.

\begin{lem}\label{EE}
Suppose $A,B\in \A$. 
The following equality holds in $\DH(\R)$. 
\[
E_{A}\ostar E_{B}
=
v^{\<A,B\>} \sum_{C\in \Iso(\R)} 
\frac{\, |\Ext^1_\R(A,B)_{C}|\, }{|\Hom_\R(A,B)|} \, E_C
\]

\end{lem}

\begin{proof}
Let $C_{A}, C_{B}$ be the complexes associated to $A, B$ in Definition \ref{def-eab}. 
Then using the formula
\[\h(C_A,C_B) = q^{\<\hat{Q}_A, \hat{P}_B\>} \cdot 
|\Hom_\R(A,B)|\]
together with the relations $\cl{A} = \hat Q_A-\hat P_A$ and 
$\cl{B} = \hat Q_B-\hat P_B$
in the Grothendieck group $K(\R)$, we have 
\begin{align*}
\gen{C_A} \ostar \gen{C_B} 
& = 
v^{\<\hat{P}_A,\hat{P}_B\>+\<\hat{Q}_A,\hat{Q}_B\>} 
\sum_{M_\blob \in \X} 
\frac{ \e(C_A,C_B)_{M_\blob}}{\h(C_A,C_B)} 
\, \gen{M_\blob}\\[1em]
&= 
v^{\<{A},{B}\>-\<\hat A,\hat P_B\> + \<\hat P_A,\hat B\> } 
\sum_{M_\blob \in \X} 
\frac{ |\Ext^1_{\C(\R)}(C_{A},C_{B})_{M_\blob}| }{|\Hom_\R(A,B)|} 
\, \gen{M_\blob}.
\end{align*}
It follows by \eqref{commute} that  
\begin{align}\label{E_mult}
E_{A}\ostar E_{B}
& = v^{ \<\hat{P}_A,\hat{A}\>+\<\hat{P}_B,\hat{B}\> + (\hat{P}_B,\hat{A})}
\cdot K_{-{P}_A\hat{\oplus}{P}_B}\ostar \gen{C_{A}}\ostar \gen{C_{B}} \nonumber
\\[1em]
&= v^{\<{A}, {B}\> + \<\hat{P}_A  + \hat{P}_B ,\, \hat{A} + \hat{B}\>} 
\sum_{M_\blob \in \X} 
\frac{ |\Ext^1_{\C(\R)}(C_{A},C_{B})_{M_\blob}| }{|\Hom_\R(A,B)|} 
\, K_{-{P}_A\hat{\oplus}{P}_B} \ostar \gen{M_\blob}.
\end{align}

Consider an extension 
\begin{equation}\label{M_ext}
0 \lra C_{B} \lra M_\blob \lra C_{A} \lra 0. 
\end{equation}
By Lemma \ref{ext_hom}, 
we may assume $M_\blob = \mathrm{Cone}(s_\blob)$ for some 
morphism $s_\blob: C_{A} \to C_{B}^\dag$, 
so that 
\[M_\blob =  \cplx{P_B\oplus P_A}{d_1}{d_0}{Q_B\oplus Q_A}, \]
where 
\[ d_1=\begin{pmatrix} f_B & s_1\\
0 & f_A \end{pmatrix}, 
\qquad
d_0=\begin{pmatrix} 0 & s_0\\
0 & 0 \end{pmatrix}. \]
Since $f_A, f_B$ are monomorphisms, so is $d_1$. 
Thus $d_1\circ d_0 =0$ implies that $s_0 = 0$. 
Setting  $C = H_0(M_\blob)$,  
it follows that \eqref{M_ext} induces an extension 
\[0\lra B \lra C \lra A \lra 0.\]
One may check that this extension agrees with 
the corresponding 
image of \eqref{M_ext} under the isomorphism 
\[\Ext^1_{\C(\R)}(C_{A},C_{B})\cong\Ext^1_\R(A,B) \]
given by Lemma \ref{ext_hom} and Lemma \ref{ext_hom2} (ii). 
It follows that
\[\sum_{H_0(M_\blob) = C} 
 |\Ext^1_{\C(\R)}(C_{A},C_{B})_{M_\blob}|
\ = \
 |\Ext^1_\R({A},{B})_{C}|.\]

Finally, notice that $\hat{C} = \hat{A}+ \hat{B}$ for 
any extension $C$ of $A$ by $B$.
Putting everything together shows that equation \eqref{E_mult} becomes 
 \begin{align*}
E_{A}\ostar E_{B}
&= v^{\<A,B\> }
\sum_{M_\blob\in \mathcal{X}} 
v^{\<\hat{P}_A  + \hat{P}_B ,\, \hat H_0(M_\blob) \>} 
\cdot 
\frac{ |\Ext^1_{\C(\R)}(C_{A},C_{B})_{M_\blob}| }{|\Hom_\R(A,B)|} 
\, K_{-{P}_A\hat{\oplus}{P}_B}\ostar
 \gen{M_\blob}
\\[1em]
&= v^{\<A,B\>} \sum_{C\in \Iso(\R)} 
\frac{ |\Ext^1_{\R}(A,B)_{C}| }{|\Hom_\R(A,B)|} \, E_{C}
\end{align*}
which completes the proof. 
\end{proof}

\begin{lem}\label{EF}
Let $A,B\in \A$. 
The following equations hold in $\DH(\R)$, 
 \smallskip
\begin{enumerate}[(i)]
\item $\displaystyle E_A \ostar F_B 
=  \sum_{A_1,B_1,B_2} 
v^{\<\cl{B}-\cl{B}_1 ,\, \cl{A} - \cl{B}\> } \, 
 g^{B}_{B_1,B_2}\, g^{A}_{B_2,A_1} a_{B_2} \cdot 
 K^\dag_{\cl{B}-\cl{B}_1} \ostar E_{A_1, B_1},$
\bigskip

\item $\displaystyle F_B \ostar E_A 
=  \sum_{A_1,A_2,B_1} 
 v^{\<\cl{A} - \cl{A}_1 ,\, \cl{B} - \cl{A}\> } \, 
 g^{A}_{A_1,A_2} g^{B}_{A_2,B_1} a_{A_2} \cdot 
K_{\cl{A}-\cl{A}_1} \ostar  E_{A_1,B_1},$
\end{enumerate}
where each sum runs over classes of objects 
in $\Iso(\A)$. 
\end{lem}

\begin{proof}
(i) Again let $C_{A}, C_{B}$ be complexes associated to $A, B$ as in Definition \ref{def-eab}. 
By definition, the product 
$\gen{C_{A}} \ostar \gen{C^\dag_{B}}$ is equal to  
\begin{equation*}
v^{\<\hat Q_A, \hat P_B \>+\<\hat P_A,\hat Q_B\>} 
\sum_{M_\blob \in \X} 
\dfrac{\e(C_{A},C_{B}^\dag)_{M_\blob} } {\h(C_{A},C_{B}^\dag)} 
\, \gen{M_\blob}. 
\end{equation*}
Then 
 using 
$\mu(C_{A},C_{B}^\dag) = \<\hat P_A,\hat Q_B\>$,
 it is easy to check that 
\[\h(C_{A},C_{B}^\dag) 
= q^{\<\hat P_A,\hat Q_B\>}.
\]
%
This gives 
\begin{equation*}
\gen{C_{A}} \ostar \gen{C^\dag_{B}} 
 = v^{\<\hat A, \hat P_B \>-\<\hat P_A,\hat B\>} 
\sum_{M_\blob \in \X} 
\e(C_{A},C_{B}^\dag)_{M_\blob}  \cdot\gen{M_\blob}, 
\end{equation*}
where we have used the equalities  
$\cl{Q}_A = \cl{P}_A + \cl{A}$ and 
$\cl{Q}_B = \cl{P}_B + \cl{B}$ in 
$K(\R)$.
It thus follows by \eqref{commute} that  
\begin{align}\label{prod}
E_A \ostar F_B 
&=
v^{\<\hat P_A,\hat A \>+\<\hat P_B,\hat B\>-(\hat P_B,\hat A)} \cdot 
K_{-\cl{P}_A}\ostar K_{-\cl{P}_B}^\dag\ostar \gen{C_{A}} \ostar \gen{C^\dag_{B}} \nonumber
\\[3pt]
&= 
v^{\<\cl{P}_A -\cl{P}_B ,\cl{A}-\cl{B}\> } \cdot K_{-\cl{P}_A}\ostar K_{-\cl{P}_B}^\dag \ostar
\sum_{M_\blob \in \X} 
\e(C_{A},C_{B}^\dag)_{M_\blob}  \cdot
\gen{M_\blob}. 
\end{align}

Now suppose that $M_\blob$ is an extension of $C_{A}$ by $C_{B}^\dag$. 
 By Lemma \ref{ext_hom}, we may assume that $M_\blob= \mathrm{Cone}(s_\blob)$ 
for some $s_\blob \in \Hom_{\Ho(\R)}(C_A,C_B)$. 
The extension thus takes the form  
\[ \begin{tikzcd}[row sep = large]
      Q_{B} \rar["0", shift left] \dar["i_1"'] & P_{B} \lar["-f_{B}", shift left]  \dar["i_0"] 
   \\
      Q_{B} \oplus P_{A}  \rar["f_1", shift left] \dar["p_1"'] 
    & P_{B} \oplus Q_{A}  \lar["f_0", shift left] \dar["p_0"]
   \\
      P_{A} \rar[ "f_{A}", shift left] &  Q_{A} \lar["0", shift left]
\end{tikzcd} \]
where 
$f_0=\left(\begin{smallmatrix}\hspace{-.075cm}-f_{B}&s_{0}\\ \ 0&0\end{smallmatrix}\right)$, 
$f_1=\left(\begin{smallmatrix}0&s_{1}\\0&f_{A}\end{smallmatrix}\right)$, and $f_A, f_B$ are defined in \eqref{eqn-resolpqa}.
This extension induces an exact commutative diagram 
\begin{equation}\label{diagram}
\begin{tikzcd}[column sep=scriptsize]
	&  0  \rar \dar & H_0(M_\blob) \rar \dar  & H_0(C_A) \dar
\\
	&  P_B  \rar["{i}_0"] \dar["-f_B"']  &  M_0/ \im f_1  \rar["{p}_0"] \dar["f_0"]	& Q_A/ \im f_A  \rar \dar &  0
\\ 
0  \rar  &  Q_B  \rar["{i}_1"] \dar 	&  \ker f_1  \rar["p_1"] \dar  &  0  \dar
\\
	&  H_1(C_B^\dag)  \rar  &  H_1(M_\blob)  \rar  &  0,
\end{tikzcd}
\end{equation}
where ``$i_0$" denotes the map induced by $i_0$, etc. 

Since the map induced by ${i}_1$ in \eqref{diagram} is an isomorphism,  
it follows that the direct summands in the 
 decomposition $M_\blob = M_\blob^+ \oplus M_\blob^-$ of Lemma \ref{decomp}
have the form 
\[M_\blob^+ = (\begin{tikzcd}[cramped, column sep=scriptsize] P_A \ar[r, "f^+_1", shift left] & M_0^+ \ar[l, "0", shift left] 
\end{tikzcd}), \quad 
M_\blob^- = (\begin{tikzcd}[cramped, column sep=scriptsize] Q_B \ar[r, "0", shift left] & M_0^- \ar[l, "f^-_0", shift left] 
\end{tikzcd}) \]
where the maps $f_0^-,f_1^+$ are obtained from $f_0,f_1$ by restriction. 

The objects $A_1 = H_0(M_\blob)$ and $B_1=H_1(M_\blob)$ thus have projective resolutions 
\[\begin{tikzcd}[column sep=2em]
 P_A \rar["f_1^+"]  &  M_0^+ \rar   &  A_1 \rar  & 0, 
\end{tikzcd}
\qquad 
\begin{tikzcd}[column sep=2em]
M_0^- \rar[ "f_0^-" ]  &  Q_B \rar  &  B_1 \rar  &  0 
\end{tikzcd}\] 
respectively. 
This gives relations 
\[\cl{M_1^+} = \cl{P_A}, \qquad \quad \cl{M_0^-} = \cl{Q_B} - \cl{B_1}
= \cl{P_B} + \cl{B} - \cl{B_1}\] 
in $K(\R)$. 
Substituting in  \eqref{prod}, we have 
\begin{align}
\label{ree}
E_A \ostar F_B \ =& \, 
\sum_{M_\blob \in \X} 
v^{\<\cl{M_1^+} -\cl{M_0^-}  + \cl{B} - \cl{B_1},\cl{A} - \cl{B}\>} \cdot 
\e(C_{A},C_{B}^\dag)_{M_\blob}  \cdot
K_{-\cl{M_0^+} } \ostar 
K_{-\cl{M_0^-}  + \cl{B} - \cl{B_1}}^\dag
\ostar
\gen{M_\blob} \nonumber 
\\[3pt ]
=&\, \sum_{M_\blob \in \X} 
v^{\<\cl{B} - \cl{B_1},\cl{A} - \cl{B}\>} \cdot 
\e(C_{A},C_{B}^\dag)_{M_\blob}  \cdot
K_{\cl{B} - \cl{B_1}}^\dag
\ostar
E_{M_\blob}.
\end{align}

One may check directly using \eqref{diagram} that the 
map $s:A\to B$ induced by \eqref{homology}  
coincides with the connecting homomorphism 
$H_0(C_A) \xrightarrow{\, \delta\, } H_1(C_B)$ 
in the long exact sequence of cohomology. 
In particular, note that 
$H_0(M_\blob) \simeq \ker s$, and $H_1(M_\blob) \simeq \cok s$.

Hence, we may conclude that 
\begin{equation}\label{ree1}
 \sum_{\substack{P_\blob \in \X \\[1mm] H_0(P_\blob)\simeq A_1,\, H_1(P_\blob) \simeq B_1}}
\hspace{-3em}
 \e(C_{A},C_{B}^\dag)_{P_\blob} \cdot E_{P_\blob} \ = \
|\{ h \in \Hom_\R(A,B) \, \mid\, \ker h \simeq A_1, \cok h \simeq B_1\}| 
\cdot E_{A_1,B_1}. 
\end{equation}
By the equality on  \cite[p.984]{Ya}, 
the preceding equation may be rewritten as 
\begin{equation}\label{ree2}
\sum_{\substack{ H_0(P_\blob)\simeq A_1, \\[1mm]
 H_1(P_\blob) \simeq B_1}}
\hspace{-3pt}
\e(C_{A},C_{B}^\dag)_{P_\blob} \cdot E_{P_\blob} \ =  
\sum_{B_2\in \Iso(\A)} g^{B}_{B_1,B_2} g^{A}_{B_2,A_1} a_{B_2} 
\cdot E_{A_1,B_1} .
\end{equation}
The equality in part (i) is now obtained  by combining \eqref{prod}, \eqref{ree2}  and \eqref{eq:assoc}. 

(ii) This follows by interchanging $A$ and $B$ in (i) and taking $\dag$ on both sides.
\end{proof}

\subsection{Embedding $\widetilde{\H_v}(\R)$ in $\DH(\R)$}
In this subsection we make some more precise statements about the relationships between the various Hall algebras we have been considering.

Consider the injective linear map  
$I_+\colon \widetilde{\H_v}(\R)\longinto \DH(\R)$ defined by  
\[ \gen{A} \ast K_\alpha \mapsto E_A \ostar K_\alpha,\] 
and let $\DH^+(\R) \subset \DH(\R)$ denote the image of 
this map. 

\begin{prop}
\label{embed}
The restriction of multiplication in $\DH(\R)$ makes the subspace $\DH^+(\R)$
into an associative algebra, and the embedding 
$I_+\colon \widetilde{\H_v}(\R)\longinto \DH(\R)$
restricts to an isomorphism $\widetilde{\H_v}(\R) \cong\DH^+(\R)$
of (unital) associative algebras. 
\end{prop}

\begin{proof}
The result follows from  Lemma \ref{EE}, together with a comparison of the 
relations  \eqref{extend} defining the extended Hall algebra with the relation \eqref{commute} 
in the localized Hall algebra. 
\end{proof}

Composing with the involution $\dag$ gives another embedding
\[{I}_-\colon \widetilde{\H_v}(\R)\longinto \DH(\R),\]
defined by $\gen{B} \ast K_\beta \mapsto F_B \ostar K_\beta^\dag$, 
whose image $\DH^-(\R)$ is again an associative algebra such that 
%
$I_-$ restricts to an algebra isomorphism $\widetilde{H_v}(\R)\cong \DH^-(\R)$. 

\subsection{Drinfeld double of $\widetilde{\H_v}(\R)$ } \label{subsec-Dri}

\begin{lem}
\label{ident}
The multiplication map $\nabla\colon a\tensor b \mapsto {I}_+(a) \ostar {I}_-(b)$ defines  an isomorphism of vector spaces 
\[\nabla\colon \widetilde{\H_v}(\R) \tensor_\CC \widetilde{\H_v}(\R) \lra \DH(\R).\]
\end{lem}

\begin{proof}
It follows from Proposition \ref{free} 
that the algebra $\DH(\R)$ has a $\CC$-basis consisting of elements 
\[E_{A,B} \ostar K_\alpha \ostar 
K^\dag_\beta, \quad A,B\in \Iso(\A), \quad \alpha,\beta\in K(\R).\]
Recall the 
partial order on $K(\A)$  defined in \eqref{order} 
and define  
$\DH_{\leq \gamma}$  for $\gamma \in K(\A)$  to be the subspace  of $\DH(\R)$ spanned by elements from this basis for which 
$\bc_\A(A)+\bc_\A(B)\leq \gamma$. 
We claim that 
\begin{equation}
\label{filter}
\DH_{\leq \gamma} \ostar \DH_{\leq \delta}
\subset \DH_{\leq \gamma+\delta}, \qquad \gamma,\delta \in K(\A),
\end{equation}
so that this defines a filtration on $\DH(\R)$.

Suppose that $M_\blob, N_\blob \in \C_\fin(\P)$ and let  
\[\gamma = 
\bc_\A(H_0(M_\blob)) + \bc_\A(H_1(M_\blob)), \quad 
\delta = 
\bc_\A(H_0(N_\blob)) + \bc_\A(H_1(N_\blob)).\] 
Then consider an extension of complexes 
\[0\lra M_\blob\lra P_\blob\lra N_\blob\lra 0.\]
The long exact sequence in homology can be split to give two long exact sequences
\begin{gather*}
0\lra K\lra H_0(M_\blob)\lra H_0(P_\blob)\lra H_0(N_\blob)\lra L\lra 0, \\
0\lra L\lra H_1(M_\blob)\lra H_1(P_\blob)\lra H_1(N_\blob)\lra K\lra 0
\end{gather*}
for some objects $K,L\in \A$. 
It follows that there is a relation  in $K(\A)$, 
\[\gamma + \delta = \bc_\A(H_0(P_\blob)) +  \bc_\A(H_1(P_\blob)) +2(\bc_\A(K)+\bc_\A(L))\]
which proves \eqref{filter}.

Suppose now that $N_\blob=C_A$ and $M_\blob=C_B^\dag$ for some objects $A,B\in \A$. Then $K=0$, and by Lemmas \ref{root}, \ref{ext_hom}, and \ref{ext_hom2}
\[\Ext^1_{\C(\R)}(N_\blob,M_\blob)=\Hom_\R(A,B),\]
and the extension class is completely determined by the connecting morphism 
$H_0(N_\blob)\to H_1(M_\blob)$.  
By Proposition \ref{prop-nonzero}, we therefore know 
that $\bc_\A(L)=0$ exactly when  the extension is trivial. 
It follows that in the  graded algebra associated to the filtered algebra $\DH(\R)$, one has
a relation
\[\nabla(\gen{A} * K_\alpha \tensor \gen{B} * K_\beta) = 
v^{-(\alpha, \cl{B})}  \cdot E_{A, B} \ostar K_{\alpha} \ostar K_{\beta}^\dag.\]
It follows that $\nabla$ takes a basis to a basis and is hence  an isomorphism. 
\end{proof}

As a corollary, we have 
\begin{cor}
\label{cor:sec1:basis}
The algebra $\DH(\R)$ has a linear basis consisting of elements 
\begin{align*}
E_A \ostar  K_\alpha \ostar F_B \ostar K_{\beta}^\dag, \quad
A,B \in \Iso(\A),\ \alpha,\beta \in K(\R).
\end{align*}
\end{cor}

Now we  state the main result of this section. 

\begin{thm}
\label{thm:Drinf}
Suppose that $\A\subset \R$ 
satisfies conditions \eqref{sigma1} and \eqref{sigma2}. 
Then the algebra $\DH(\R)$ is isomorphic to the Drinfeld double of 
the bialgebra $\widetilde{\H_v}(\R)$.
\end{thm}

\begin{proof}
Because of the description 
of the basis of $\DH(\R)$ 
(Corollary \ref{cor:sec1:basis})
and the definition of Drinfeld double, 
the proof of the theorem 
is reduced to check equation \eqref{eq:Drinf} 
for the elements 
consisting of the basis of $\widetilde{\H_v}(\R)$.

Let us write equation \eqref{eq:Drinf} in the present situation:
\begin{equation}
\label{eq:dd1}
  \sum (a_{2},b_{1})_H \cdot 
  {I}_{+}(a_{1}) \ostar {I}_{-}(b_{2})
\ \stackrel{?}{=}\ 
  \sum (a_{1},b_{2})_H \cdot 
   {I}_{-}(b_{1}) \ostar {I}_{+}(a_{2}).
\end{equation}
Now let $A,B \in \A$ and $\alpha, \beta \in K(\R)$.  Let us write  
\begin{align*}
 \Delta([A] K_\alpha)&=
  \sum_{A_1,A_2} v^{\<A_1,A_2\>} 
  g_{A_1,A_2}^A   \cdot 
   ([A_1] K_{\cl{A_2}+\alpha})\otimes ([A_2] K_\alpha),
\\
 \Delta([B] K_\beta)&=
  \sum_{B_2,B_1} v^{\<B_2,B_1\>} 
  g_{B_2,B_1}^B  \cdot 
   ([B_2] K_{\cl{B_1}+\beta}) \otimes ([B_1]  K_\beta).
\end{align*}

By the Hopf pairing (Definition \ref{df:hp} and Proposition \ref{prop-Hopf-pairing}) 
and \eqref{commute}, the left hand side of \eqref{eq:dd1} becomes
\begin{align*}
 \text{LHS of } \eqref{eq:dd1}
&=\sum_{A_1,A_2,B_1,B_2}
   v^{\<A_1,A_2\>} 
  g_{A_1,A_2}^A 
   v^{\<B_2,B_1\>} 
  g_{B_2,B_1}^B
     ([A_2]  K_{\alpha},[B_2]  K_{\cl{B_1}+\beta} )_H  \\
& \hspace{3cm}     \cdot 
   E_{A_1} \ostar K_{\cl{A_2} + \alpha} \ostar F_{B_1} \ostar K^\dag_{\beta}
\\[.5cm]
=\ v^{(\alpha,\beta)} &\sum_{A_1,A_2,B_1,B_2}
   v^{\<A_1,A_2\>+\<B_2,B_1\>} 
  g_{A_1,A_2}^A 
  g_{B_2,B_1}^B
([A_2],[B_2])_H  \\
& \hspace{3cm}     \cdot 
   E_{A_1} \ostar K_{\cl{A_2}} \ostar F_{B_1} \ostar K_{\alpha} \ostar K^\dag_{\beta} .
\end{align*}
Similarly, the right hand side  becomes
\begin{align*}
& \text{RHS of } \eqref{eq:dd1} \\
&=\ v^{(\alpha,\beta)} \sum_{A_1,A_2,B_1,B_2} 
 v^{\<A_1,A_2\>+\<B_2,B_1\>} 
  g_{A_1,A_2}^A g_{B_2,B_1}^B ([A_1],[B_1])_H  \cdot    
   F_{B_2} \ostar K^\dag_{\cl{B_1}} \ostar E_{A_2} \ostar K_{\alpha} \ostar K^\dag_{\beta} . 
\end{align*}

After removing the term $v^{(\alpha,\beta)} \cdot K_{\alpha} \ostar K^\dag_{\beta}$ 
from both sides,  
equation \eqref{eq:dd1} reduces to 
\begin{equation}
\label{eq:dd2}
\begin{split}
\sum_{A_1,A_2,B_1,B_2}
&   v^{\<A_1,A_2\>+\<B_2,B_1\>} 
  g_{A_1,A_2}^A 
  g_{B_2,B_1}^B
([A_2],[B_2])_H  \cdot 
   E_{A_1} \ostar K_{\cl{A_2}} \ostar F_{B_1}
   \\
& \stackrel{?}{=}
 \sum_{A_1,A_2,B_1,B_2}
 v^{\<A_1,A_2\>+\<B_2,B_1\>} 
  g_{A_1,A_2}^A g_{B_2,B_1}^B ([A_1],[B_1])_H  \cdot    
   F_{B_2} \ostar K^\dag_{\cl{B_1}} \ostar E_{A_2}. 
\end{split}
\end{equation}
By Definition \ref{df:hp} and \eqref{commute},
the left hand side of \eqref{eq:dd2} becomes
\begin{align*}
 \text{LHS of } \eqref{eq:dd2}
&= 
 \sum_{A_1,A_2,B_1,B_2}
   v^{\<A_1,A_2\>+\<B_2,B_1\>} 
  g_{A_1,A_2}^A 
  g_{B_2,B_1}^B
a_{A_2} \delta_{A_2,B_2} \cdot 
   E_{A_1} \ostar K_{\cl{A_2} } \ostar F_{B_1} 
\\
&= 
\sum_{A_1,A_2,B_1}
   v^{\<A_1,A_2\>+\<A_2,B_1\>} 
  g_{A_1,A_2}^A 
  g_{A_2,B_1}^B
  a_{A_2}  \cdot 
   E_{A_1} \ostar K_{\cl{A_2}} \ostar F_{B_1}\\
&= 
\sum_{A_1,A_2,B_1}
   v^{\<\cl{A_2},\cl{B_1}\> -  \<\cl{A_2},\cl{A_1}\>} 
  g_{A_1,A_2}^A 
  g_{A_2,B_1}^B
  a_{A_2}  \cdot 
   K_{\cl{A_2}} \ostar E_{A_1} \ostar F_{B_1}.
\end{align*}
Thus by Lemma \ref{EF} (i) we have
\begin{align*}
 \text{LHS of } \eqref{eq:dd2}
&= 
\sum_{A_1,A_2, A_3, B_1, B_2,B_3}
v^{\< \cl{B_2} - \cl{A_2} , \, \cl{A} - \cl{B}\> } 
  g^{A_1}_{B_2,A_3} 
  g_{A_1,A_2}^A 
    g_{A_2,B_1}^B 
  g^{B_1}_{B_3,B_2} 
  \cdot a_{A_2}   a_{B_2}
 \\
& 
\hspace{7cm}
\cdot 
   K_{\cl{A_2}} \ostar K^\dag_{\cl{B_2} } \ostar 
   E_{A_3,B_3} .
\end{align*}


Similar computations using Lemma \ref{EF} (ii) show that 
the right hand side of \eqref{eq:dd2} becomes 
\begin{align*}
 \text{RHS of } \eqref{eq:dd2}
= 
 \sum_{A_1,A_2,A_3,B_1,B_2,B_3} &
 v^{\<\cl{A_1} - \cl{B_1},\, \cl{B} - \cl{A}\> }  
  g^{A_2}_{A_3,A_1}  g^{B_2}_{A_1,B_3}  g_{B_2,B_1}^B g_{B_1,A_2}^A
\cdot a_{B_1} a_{A_1}  \\
 &\qquad \quad \cdot
 K_{\cl{A_1}} 
\ostar K_{\cl{B_1}}^\dag \ostar E_{A_3,B_3} .
\end{align*}
It follows by associativity \eqref{eq:assoc} that this can be rewritten as 
\begin{align*}
 \text{RHS of } \eqref{eq:dd2}
=  \sum_{A_1,A_3,B_1,B_3} &
 v^{\<\cl{A_1} - \cl{B_1},\, \cl{B} - \cl{A}\> }  
 \sum_{A_2',B_2'}
  g^{A_2'}_{B_1,A_3}  g^{A}_{A_2',A_1} g^{B}_{A_1,B_2'} g^{B_2'}_{B_3,B_1} 
\cdot a_{B_1} a_{A_1}  \\
 &\qquad \quad \cdot
  K_{\cl{A_1} }  \ostar K_{\cl{B_1}}^\dag
   \ostar E_{A_3,B_3} .
\end{align*}
One may check that this expression agrees with the LHS of \eqref{eq:dd2}, 
which completes the proof.
\end{proof}

It is now possible to verify that the multiplication in $\DH(\R)$ is associative. 

\begin{cor} \label{cor-associative}
If the category $\R$ is Hom-finite (so that $\A=\R$) or if the subcategory $\A\subset \R$ 
satisfies conditions \eqref{sigma1} and \eqref{sigma2}, 
then the algebra $\DH(\R)$ is associative. 
\end{cor}

\begin{proof}
If $\R=\A$, then it follows by Remark \ref{Bridg} that the algebra $\DH(\R)$ is isomorphic 
to the localized Hall algebra $\DH(\A)$ defined in \cite{Br}.  
It follows by results in \cite{Br} that the category $\C(\A)$ is Hom-finite, 
so that $\H(\C(\P))$ and $\DH(\A)$ are both associative in this case. 

The remaining statement is a consequence of Theorem \ref{thm:Drinf}
since the multiplication in the Drinfeld double is associative. 
\end{proof}

\begin{remark}
From the above result we can only conclude that the algebra $\H(\C_\fin(\P))$ 
is ``locally associative" in general, in the sense that   
given $x,y,z \in \H(\C_\fin(\P))$ we have 
\[u\ostar 
(x \ostar (y \ostar z) - 
(x \ostar y) \ostar z) = 0 \]
for some element $u\in \Z$. 
\end{remark}

\subsection{Reduction} 
Define the reduced localized Hall algebra by setting $\gen{M_\blob}=1$ in 
 $\DH(\R)$  whenever $M_\blob$ is an acyclic complex, invariant under the shift functor.  
More formally, we set
\[\DH_\rd(\R)=\DH(\R)/\big(\gen{M_\blob}-1: H_\blob(M_\blob)=0, \  M_\blob\isom M_\blob^\dag\big).\]  
By Lemma \ref{acyclic} this is the same as setting
\begin{equation*}
\label{hot}\gen{K_{P}} \ostar \gen{K_P^\dag} = 1\end{equation*}
for all $P\in \P$.  
One can check that the shift functor $\dag$ defines involutions 
of $\DH_\rd(\R)$.

We have the following triangular decomposition.
\begin{prop}
\label{tri2}
The multiplication map $\gen{A}\tensor \alpha\tensor \gen{B}\mapsto E_A \ostar K_\alpha \ostar F_B$ defines
 an isomorphism of vector spaces
\[\H(\A)\tensor_{\CC} \CC[K(\R)]\tensor_{\CC} \H(\A) \lra \DH_{\rd}(\R).\]
\end{prop}

\begin{proof}
The same argument given for Lemma \ref{ident} also applies here.
\end{proof}

\subsection{Commutation relations} \label{subsec-comm-rels}

In this subsection, we prove commutation relations among generators that are important to understand $\DH_{\rd}(\R)$.

\begin{lem}
\label{commute2}
Suppose $A_1,A_2\in \A$ satisfy
\begin{equation*}
\Hom_\R(A_1,A_2)=0=\Hom_\R(A_2,A_1).
\end{equation*}
Then
$[E_{A_1} ,F_{A_2}]=0$.
\end{lem}

\begin{proof}
It follows from \eqref{ree} and \eqref{ree1} 
that $E_{A_1} \ostar F_{A_2} = E_{A_1,A_2}$.  
Exchanging $A_1$ and $A_2$ in this equation and taking $\dag$ on both sides  
gives 
$F_{A_2} \ostar E_{A_1} = E_{A_1,A_2}$ as well, and the result follows. 
\end{proof}

\begin{lem}
\label{EF-rel}
Suppose $A\in\A$ satisfies $\operatorname{End}_\A(A)=\k$.
Then
\[[E_A ,F_A] =  (q-1)\cdot (K_{\cl{A}}^\dag - K_{\cl{A}}).\]
\end{lem}

\begin{proof}
Using formulas \eqref{ree} and \eqref{ree1} again, we have 
$E_A\ostar F_A = E_{A,A} + (q-1)\cdot K_A^\dag$. 
Taking $\dag$ on both sides gives the equation 
$F_A\ostar E_A = E_{A,A} + (q-1)\cdot K_A$, 
since $E_{A,A}$ is $\dag$-invariant. The result follows by subtracting these equations. 
\end{proof}

\section{Realization of quantum groups} \label{sec-q-group}

\subsection{Quivers}

Let $\mathcal Q$ be a locally
finite quiver with vertex set $I$ and (oriented) edge set $\Omega$. For
$\sigma \in \Omega$
we denote by $h(\sigma)$ and $t(\sigma)$ the head and tail,  respectively,
and sometimes use the notation $t(\sigma) \stackrel{\sigma}{\to}
h(\sigma)$.  We will denote by
$c_i$ the number of loops at $i \in I$ (i.e., the number of edges $\sigma$
with  $h(\sigma)=t(\sigma)=\,i$).
A (finite) 
{\em path} in $\mathcal Q$ is a sequence 
$ \sigma_m \cdots \sigma_1\,$ 
of edges which satisfies $h(\sigma_{i}) = t(\sigma_{i+1})$ for $1\leq i < m$. 
For each $i\in I$, we let $e_i$ denote the trivial path. 
We again let $h(x)$ and $t(x)$ denote the head and tail vertices of a path $x$.

Consider the sub-quiver $\bar{\mathcal Q }\subset \mathcal Q$ with vertex set $\bar{I}=I$ and edge set 
$\bar{\Omega} = \Omega \backslash \{ \sigma | h(\sigma)=t(\sigma) \}$. 
We make the following assumptions throughout: 
\begin{itemize}
\item[(A)] There are no infinite paths of the form $i_0 \rightarrow i_1 \rightarrow i_2 \rightarrow \cdots$ in $\bar{\mathcal Q}$. In particular, $\bar{\mathcal Q}$ is acyclic and $\mathcal Q$ has no oriented cycles other than loops.

\item[(B)] Each vertex of the quiver  $\mathcal Q$ has either zero loops or more than one loop, i.e.  
$c_i\neq 1$ for all $i\in I$. 
\end{itemize}
It follows from (A) that  $I$ is partially ordered, with $i\preceq j$ if there exists  a path $x$ 
such that $t(x)=j$ and $h(x)=i$, 
and the set $(I, \preceq)$ satisfies the descending chain condition. 

From now on,  we assume that the quiver $\mathcal Q$  satisfies the conditions (A) and (B).

\begin{example}
Write $\L_n$ to denote the quiver consisting of a single vertex $I=\{1\}$ and $n$ loops. If $n\ge 2$, then the quiver $\mathcal Q= \L_n$ trivially satisfies the assumptions (A) and (B).
Below is a diagram for $\L_4$.
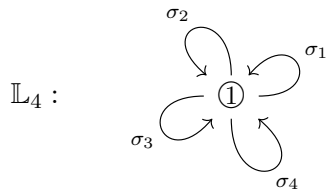
\[\L_4 : \qquad
\begin{tikzcd}
\raisebox{.5pt}{\textcircled{\raisebox{-.9pt} {\small 1}}} 
\arrow[out=0,in=50,loop,swap,"\sigma_1"]
  \arrow[out=90,in=140,loop,swap,"\sigma_2"]
  \arrow[out=180,in=230,loop,swap,"\sigma_3"]
  \arrow[out=270,in=320,loop,swap,"\sigma_4"]
\end{tikzcd}\]
\end{example}

Let $R=\k \mathcal Q$ 
denote the path algebra, with basis given by the set of paths 
and multiplication defined via concatenation.  
The elements $e_i$ are pairwise orthogonal idempotents. 
It follows from (A) that the subring $e_i R e_i$ is isomorphic to a free associative 
$\k$-algebra on $c_i$ generators. (If $c_i=0$ then $e_i R e_i \cong \k$.) 
Hence each left ideal $P_{i}:=  Re_i$  is an indecomposable projective $R$-module. 
It can then be checked that 
$R$ splits as a direct sum 
\begin{equation}\label{princ}
R  = \bigoplus_{i\in I} P_i
\end{equation}
of pairwise non-isomorphic projective left $R$-submodules. 

A representation of $\mathcal Q$ over $\k$ is a collection $(V_i, x_{\sigma})_{i \in I, \sigma \in \Omega}$, where $V_i$
is a (possibly infinite dimensional) $\k$-vector space and
$x_{\sigma} \in \mathrm{Hom}_\k (V_{t(\sigma)}, V_{h(\sigma)})$.  
We let $\mathrm{Rep}_{\k}(\mathcal Q)$ denote 
the abelian category consisting of the representations of $\mathcal Q$ over $\k$ which are 
of {\em finite support}, i.e.~ such that $V_i = 0$ for almost all $i$. 
A representation $(V_i, x_{\sigma}) \in \mathrm{Rep}_{\k}(\mathcal Q)$ is called {\em finite-dimensional} if 
each $V_i$ is finite-dimensional. 
For such a representation, set $\mathbf{dim}(V_i,x_{\sigma})=
\mathbf{dim}(V_i)=(\dim_{\k} V_i) \in \N^{\oplus I}$. 
We denote by $\mathrm{rep}_{\k}(\mathcal Q)$ the full subcategory of $\R$ consisting of the finite dimensional representations of $\mathcal Q$.
 
Any representation of $\mathcal Q$ is naturally an $R$-module 
for the path algebra $R$. Whenever it is convenient, particularly in the next subsection, we will consider representations of $\mathcal Q$ as $R$-modules. 
It follows from the decomposition \eqref{princ} that any left $R$-module $M$ has a decomposition into $\k$-subspaces,  
$M= \bigoplus_{i\in I} M(i) $, with $M(i) = e_iM$. 
Then we have $\mathbf{dim}(M)= (\dim_{\k} M(i)) \in \N^{\oplus I}$ for a finite dimensional $R$-module $M$, which is equal to the dimension vector as a representation.  We say that an $R$-module is of finite support if the associated representation is. 

Recall that any $M \in \mathrm{Rep}_{\k}(\mathcal Q)$ has the standard presentation 
of the form 
\begin{equation} \label{std}
0 \lra \bigoplus_{\sigma \in \Omega} P_{h(\sigma)} \otimes_\k e_{t(\sigma)} M 
\xrightarrow{ \ f\  } 
\bigoplus_{i\in I} P_i \otimes_\k e_i  M 
\xrightarrow{ \ g\ } M \lra 0. 
\end{equation}
Let $\mathrm{Proj}_\k(\Q)$ denote the full subcategory of $\mathrm{Rep}_{\k}(\mathcal Q)$ whose objects are 
finitely-generated, projective $R$-modules. 
Let $\R$ be the full subcategory of $\mathrm{Rep}_{\k}(\mathcal Q)$ whose objects are finitely presented representations of $\Q$, i.e. the full subcategory 
consisting of all 
objects $M$ for which there exists a presentation 
\[ P \to Q \to M \to 0 \]
for some $P,Q \in \text{Proj}_\k(\Q)$.

As in Section \ref{ass}, define $\P\subset \R$ to be the 
full subcategory of projectives in $\R$ and 
$\A \subset \R$ the full subcategory  
of objects $A\in \R$ such that $\Hom_\R(M,A)$ is a finite set for any $M\in \R$. It is easy to see that $\P = \mathrm{Proj}_\k(\Q)$. For the category $\A$, we have the following characterization.

\begin{lem} \label{lem-finite-sub}
The category $\A$ is equal to the full subcategory of $\R$ consisting of all 
finite-dimensional representations, i.e. $\A = \mathrm{rep}_\k(\mathcal Q)$. 
\end{lem}

\begin{proof}
Assume that $A \in \mathrm{rep}_\k(\mathcal Q)$. From the standard resolution \eqref{std}, we see that $A \in \R$. Clearly, $A$ is finite as a set. Since any $M \in \R$ is finitely generated, the set $\Hom_\R(M,A)$ is also finite. Thus $A$ is an object of $\A$. For the converse, assume that $M \in \mathrm{Rep}_\k(\mathcal Q)$ is infinite dimensional. Then there is a vertex $i \in I$ such that $e_i M$ is infinite dimensional. For each $a \in e_i M$, we have a homomorphism $P_i \rightarrow M$ given by $e_i \mapsto a$. Thus  $\Hom_\R(P_i,M)$ is an infinite set and $M$ does not belong to $\A$. 
\end{proof}

\subsection{Krull--Schmidt property for $\mathrm{Proj}_\k(\Q)$} 

Since the endomorphism ring  $\End_R( P_i)$ is not local in general, 
the usual Krull--Schmidt theorem 
does not hold in the category   
$\mathrm{Proj}_\k(\Q)$. 
In this subsection, 
we describe a suitable analogue. 

First note the following. 

\begin{lem}\label{prec}
Suppose $i,j\in I$ are distinct vertices such that  
$(R e_i R) \cap (R e_j R)  \neq 0$. 
Then either $i\prec j$ or $j\prec i$. 
\end{lem}

\begin{proof}
It follows from the stated condition
that there are paths $x,x',y,y'$ such that 
$x e_i y$ and $x' e_j y'$ are both nonzero and 
\[
xy = x e_i y  =  x' e_j y' =  x'y'.
\]
We then have 
$xy =  x'y' = 
\sigma_1 \cdots \sigma_n$, for some $\sigma_1, \dots, \sigma_n \in \Omega$. 
Let $1\leq l, l' \leq n$ be such that 
$x = \sigma_1 \cdots \sigma_l$, $y= \sigma_{l+1} \cdots \sigma_n$, $x' = \sigma_1 \cdots \sigma_{l'}$, 
and $y'= \sigma_{l'+1} \cdots \sigma_n$. 
Suppose without loss of generality that $l<l'$.  Then $z=\sigma_{l+1}\cdots \sigma_{l'}$ is a path such that 
$h(z)= i$ and $t(z) = j$.  
Thus $i \prec j$.  
\end{proof}

Let $R$-Mod denote the abelian category of all left $R$-modules of finite support. We identify $R$-Mod with $\mathrm{Rep}_\k(\Q)$, and consider $\P=\mathrm{Proj}_\k(\Q)$ as a full subcategory of $R$-Mod.
We write
$\mathrm{supp}(M):= \{i \mid M(i)\neq 0\}$ for $M \in R$-Mod. 
Recall that for any idempotent $e\in R$ there is an exact functor 
from the category $R$-Mod to $(eRe)$-Mod given by $M\mapsto eM$.   
Now suppose $M \in \P$. 
Then  $M(i)$ is a finitely-generated, projective 
$(e_i R e_i)$-module. 
Since $e_i R e_i$ is a free associative $\k$-algebra, 
$M(i)$ is a free $(e_i R e_i)$-module of finite rank, 
say $r_i(M)$. (See, for example, \cite{Cohn}.)
Write $\mathbf{rk}_\P(M) = (r_i(M))_{i\in I}\in \mathbf{N}^{\oplus I}$ to denote the vector formed by these ranks. 
Notice that $r_i(P_i) = 1 $ and $r_j(P_i) = 0$ unless $j\preceq i$. 
It follows that the vectors, $\mathbf{rk}_\P(P_i)$, form a basis 
for $\ZZ^{\oplus I}$.

Given a subset $J\subseteq I$, consider the set $J_{\preceq} = \{ i \in I  \mid i\preceq j \text{ for some } j\in J\}$.  
If $J$ is finite then so is $J_{\preceq}$ by (A), and there are corresponding idempotents 
\[
e_J = \sum_{j\in J} e_j \qquad \text{ and } \qquad
e_{\preceq J} =  \sum_{ i \in J_{\preceq} } e_i.
\]
%

\begin{lem}\label{equiv}
Suppose $J\subseteq I$ is a finite subset 
and write $e=e_{\preceq J}$.
Then there is an equivalence between $(eRe)\text{\rm -Mod}$ and the full subcategory of $R\text{\rm -Mod}$ 
consisting of modules $M$ such that 
$\mathrm{supp}(M)\subseteq J$.  
\end{lem}

\begin{proof}
Consider the functor 
$R\text{-Mod} \to (eRe)\text{-Mod}: M\mapsto eM$.   
Then $eRe = Re$, and an inverse functor is given by extending the action of $Re$ on 
$M \in (eRe)$-Mod to all of $R$ by letting  
$\bigoplus_{i\notin J_{\preceq}} P_i$  act by zero. 
\end{proof}

\begin{lem}\label{lem:KS}
Suppose $\mathfrak a $ is a finitely generated left ideal of $R$. For $i,j\in I$, the following hold.   
\begin{enumerate}[(i)]
\item The ideal $\mathfrak a $ is a projective left $R$-module.  
\\[-.3cm]
\item Any nonzero $R$-module homomorphism, $\phi: P_i \to \mathfrak a$, is injective. 
\\[-.3cm]
\item Suppose 
$\phi_i: P_i \to \mathfrak a$ and 
$\phi_j: P_j \to \mathfrak a$ 
are 
homomorphisms such that 
$\im(\phi_i)\cap \im(\phi_j) \neq 0$.
Then either 
$\im(\phi_i)\subseteq \im(\phi_j)$
or
$\im(\phi_j)\subseteq \im(\phi_i)$.
\\[-.3cm]
\item As a left $R$-module, $\mathfrak a$ is isomorphic to a finite direct sum 
of copies of the modules $\{P_i\}_{i\in I}$. 

\end{enumerate}
\end{lem}

\begin{proof}
(i) Suppose $\mathfrak a\subseteq R$ is a left ideal with $S\subseteq \mathfrak a$ a finite set of generators. 
Then $S \subseteq Re_J$ for some finite set $J$ and it 
follows 
that $\mathrm{supp}(\mathfrak a)\subseteq J_{\preceq}$. 
If we set $e=e_{\preceq J}$, 
then $\mathfrak a\subseteq e R e$, which shows that $\mathfrak a$ is a left ideal of a hereditary ring  
and thus projective as an $eRe$-module. 
It follows that $\mathfrak a$ is a projective $R$-module 
by Lemma \ref{equiv}.

(ii)  The image $\im(\phi) \subseteq \mathfrak a$ is a projective left $R$-module by (i). 
So the exact sequence 
\[0\to \ker(\phi) \to P_i \to \im(\phi) \to 0 \]
splits. 
Since 
$P_i$ is indecomposable, it follows that $\phi$ is injective.

(iii) Letting $\phi_i(e_i)=v$ and $\phi_j(e_j)=w$,  
we have $\im(\phi_i) = Rv$ and  $\im(\phi_j) = Rw$. 
%
%
We thus have $(Re_iv)\cap(Re_jw) \neq 0$.  It follows by Lemma \ref{prec} 
that $i\preceq j$ or $j\preceq i$.  Assume without loss of generality that $j\preceq i$.  
It then follows from the proof of Lemma \ref{prec}, that there exists a path $x$ 
such that $w = e_jw = e_j x e_i v = xv$.  It follows that 
$\im(\phi_j) = R w = R (xv) \subseteq R v =\im(\phi_i)$.

(iv)
First set $J^1 := \mathrm{supp}(\mathfrak a)$, 
and choose a maximal vertex $j_1\in J^1$. 
Then $\mathfrak a_{j_1}$ is an $e_{j_1}R{e_{j_1}}$-module with a  
finite set, say $\mathcal{S}^1$, of free generators of size 
$n_1 := r_{j_1}(\mathfrak a) \neq 0$. 
It follows from (ii) that for each generator $v\in \mathcal{S}^1$ the mapping,     
$P_{j_1} \to \mathfrak a: e_{j_1}\mapsto v$,
defines an injective $R$-module homomorphism.  
By (iii), we thus have a corresponding isomorphism 
\[ (P_{j_1})^{\oplus n_1}\, \lRa{\sim} R\mathcal{S}^1\  \subseteq\, \mathfrak a.\]

Next choose a maximal vertex $j_2$ belonging to the subset 
\[ J^2 := \{ j\in J^1\mid r_{j}(\mathfrak a) - n_1\cdot r_{j}(P_{j_1})>0\}.\]
It follows that $\mathfrak a_{j_2} = e_{j_2} \mathfrak a$ is a free $e_{j_2} R e_{j_2}$-module of 
rank 
$n_1\cdot r_{j}(P_{j_1}) + n_2$, 
for some $n_2 >0$. 
Let $\mathcal{S}^2 = \mathcal{S}^2_1 \sqcup \mathcal{S}^2_2$ be a set of free $e_{j_2} R e_{j_2}$-generators 
such that $\mathcal{S}^2_1$ generates $e_{j_2}(R\mathcal{S}^1)$. 
Then $\mathcal{S}^2_2$ has size $n_2$, and it follows as in the previous paragraph that we have an embedding 
\[ (P_{j_2})^{\oplus n_2}\, \lRa{\sim} R\mathcal{S}^2_2\  \subseteq\, \mathfrak a.\]
By the maximality of $j_1$ we have $j_1 \npreceq j_2$.  
It is also clear that $R\mathcal{S}^2_2 \nsubseteq R\mathcal{S}^1$. 
It follows by (iii) that $(R\mathcal{S}^1)\cap (R\mathcal{S}^2_2) =0$. 
We thus obtain an embedding 
\[  (P_{j_1})^{\oplus n_1} \oplus (P_{j_2})^{\oplus n_2}\, \lRa{\sim} R\mathcal{S}^1 + R\mathcal{S}^2_2\  \subseteq\, \mathfrak a.\]
Continuing in this way the process eventually terminates, since $\mathrm{supp}(\mathfrak a)$ is finite, 
yielding the desired decomposition. 
\end{proof}

The following is an analogue of the Krull--Schmidt theorem for the category of finitely-generated projective $R$-modules.

\begin{prop}\label{prop:KS}
Given any finitely-generated, projective left $R$-module 
$M \in\mathrm{Proj}_\k(\Q)$, 
there is an $R$-module isomorphism 
\[M\cong \bigoplus_{i\in I} P_i^{\oplus n_i}\] 
for some nonnegative
integers $n_i$, only finitely many of which are nonzero.  
Moreover, given another such decomposition, $M\cong \bigoplus_{i\in I} P_i^{\oplus m_i}$, 
we must have $m_i = n_i$ for all $i$. 
\end{prop}

\begin{proof}
Since $M$ is finitely generated, $J=\mathrm{supp}(M)$ is finite.  
Letting $e=e_{\preceq J}$, we see that $M$ is a projective module of 
the hereditary ring $eRe$.  It follows from \cite[Theorem 5.3]{CE} or \cite{Kap} that 
$M$ is isomorphic to a finitely generated left ideal of $eRe$.  
Hence  Lemma \ref{lem:KS} yields a decomposition 
\[M\cong \bigoplus_{i\in I} P_i^{\oplus n_i}.\] 
It follows that 
\[\mathbf{rk}_\P(M) = \sum_{i \in I} n_i \cdot \mathbf{rk}_\P(P_i).\]
Since the vectors 
$\{\mathbf{rk}_\P(P_i)\}_{i\in I}$ form a basis 
for $\ZZ^{\oplus I}$, the decomposition must be unique. 
\end{proof}

\subsection{Assumptions (a)-(e)} \label{sec-assumptions-ae}

Let  $S_i$ be a simple module supported only at $i \in I$. Then we obtain from \eqref{std} the standard resolution 
\begin{equation*}
0 \lra P_i' \lra P_i \lra S_i \lra 0 ,
\end{equation*}
where $P'_i = 
\bigoplus_{j\in I} P_j^{\oplus n_j}$ 
for some integers $n_j$. 
Then clearly $n_j=0$ unless $j\preceq i$.   
In particular, if $i$ is a minimal vertex then 
$P'_i = P_i^{\oplus c_i}$ and hence 
\begin{equation}\label{res0}
(1-c_i) \, \cl{P}_i =  \cl{S}_i
\end{equation}
which is non-zero by assumption (B). 
Now if $i\in I$ is not minimal, then by assumption (A) the set $\{j \mid j\preceq i\}$ is finite.  
We may thus use \eqref{std} and \eqref{res0} inductively to write $\cl{P}_i$ as a linear combination 
\begin{equation}\label{class}
\cl{P}_i = \frac 1 {1-c_i } \, \cl{S}_i + \sum_{j\prec\, i} r_{ij} \, \cl{S}_j ,\qquad  r_{ij} \in \QQ .
\end{equation}

The following proposition makes it possible to apply the results in the general setting of the previous sections 
to the category $\R$ of finitely-presented quiver representations.

\begin{prop} \label{prop-ae}
The triple $(\R, \P, \A)$ satisfies the assumptions (a)-(e) in Section \ref{ass}. 

\end{prop}
\begin{proof}
(a) It is clear.
(b) Since the path algebra $R$  is hereditary, the category $\R$  is hereditary as well. Furthermore, $\R$ has enough projectives  by definition.  
(c) It is clear that $\P= \mathrm{Proj}_\k(\Q)$, so this condition follows easily from 
Proposition \ref{prop:KS}.
(d) It follows from the expression \eqref{class}.
(e) If $\bc_\R(A)=\bc_\R(B)$ for  $A, B \in \A$, the standard resolution \eqref{std} tells us that $e_iA$ and $e_iB$ have the same number of elements for each $i \in I$. Then $| \Hom(P_i,A)|=|\Hom(P_i,B)|$ for each $i \in I$, and thus $| \Hom(P,A)|=|\Hom(P,B)|$ for $P \in \P$ by Proposition \ref{prop:KS}. 
\end{proof}

It is also clear that the subcategory $\A\subset \R$ satisfies the finite subobjects condition \eqref{sub}, since each object 
$A\in \A$ is a finite dimensional vector space by Lemma \ref{lem-finite-sub}. Thus $\A$ also satisfies conditions \eqref{sigma1} and \eqref{sigma2}.

\begin{remark}

The quiver $\L_1$ is called the {\em Jordan quiver}, and its path algebra is isomorphic to the polynomial algebra $\k[x]$. For each $\lambda \in \k$, there is the one-dimensional simple modules $S_{\lambda}$ over $\k[x]$ where $x$ acts as $\lambda$. Considering the standard resolution \eqref{std}, we see $\widehat{S}_{\lambda}=0$, and the assumptions (d) and (e) are not satisfied. Nonetheless, the Jordan quiver is related to classical examples of Hall algebras. One can find details, for example, in  \cite{Sch}.

\end{remark}

\subsection{Quantum generalized Kac-Moody algebras}

In this subsection we recall the basic definitions concerning quantum generalized Kac--Moody algebras. We keep the assumptions on the choice of $v$ as in Section \ref{sec-notations}.

Let $I$ be a countable index set, and fix a symmetric {\em Borcherds--Cartan} matrix $A=( a_{ij} )_{i,j \in I}$ whose entries $a_{ij}$, by definition, satisfy (i) $a_{ii} \in \{2,0,-2,-4,\ldots\}$ and (ii) $a_{ij}=a_{ji} \in \ZZ_{\leq 0}$ for all $i,j$. Put $I^{re}=\{i \in I\;|a_{ii}=2\}$ and
$I^{im}=I\backslash I^{re}$, and  assume that we are given a
collection of positive integers $\m=(m_i)_{i\in I}$, called the {\em charge} of $A$,  with $m_i=1$ whenever $i \in I^{re}$. We put
$$[n]=\frac{v^n-v^{-n}}{v-v^{-1}}, \qquad [n]!=[1][2] \cdots [n], \qquad
\left[\begin{matrix} n\\k\end{matrix}\right]=\frac{[n]!}{[n-k]![k]!}.$$

The {\it quantum generalized Kac--Moody algebra}
associated with $(A,\m)$ is defined to be the
(unital) $\CC$-algebra $\U_v$
 generated by the elements $K_i,K_i^{-1},
E_{ik}, F_{ik}$ for $i \in I$, $k=1, \ldots, m_i$, subject to the
following  set of relations: for $i,j \in I$, $k=1, \dots, m_i$ and $l=1, \dots, m_j$,

\begin{equation}\label{E:01}
K_iK_i^{-1}=K_i^{-1}K_i=1, \quad K_iK_j=K_jK_i, 
\end{equation}
\begin{equation}\label{E:02}
K_iE_{jl}K_i^{-1}=v^{a_{ij}}E_{jl}, \qquad K_iF_{jl}K_i^{-1}=
v^{-a_{ij}}F_{jl},
\end{equation}
\begin{equation}\label{E:03}
E_{ik}F_{jl} - F_{jl}E_{ik} = \delta_{lk}\delta_{ij}
\frac{K_i-K_i^{-1}}{v-v^{-1}},
\end{equation}
\begin{equation}\label{E:05}
E_{ik}E_{jl} - E_{jl}E_{ik}=F_{ik}F_{jl} - F_{jl}F_{ik} =0 \qquad 
\text{if} \
a_{ij}=0, \quad \text{ and}\end{equation}
\begin{align}\label{E:04}
&\sum_{n=0}^{1-a_{ij}}(-1)^n
\left[\begin{matrix} 1-a_{ij}\\n\end{matrix}\right] E_{ik}^{1-a_{ij}-n}
E_{jl}E^n_{ik}\\
&=\sum_{n=0}^{1-a_{ij}}(-1)^n
\left[\begin{matrix} 1-a_{ij}\\n\end{matrix}\right] F_{ik}^{1-a_{ij}-n}
F_{jl}F^n_{ik}=0 \qquad \text{if }  i \in I^{re} \text{ and } i \neq j. \nonumber
\end{align}
The algebra $\U_v$ is equipped with a Hopf algebra
structure as follows (see \cite{BKM, Ka95}):
$$\Delta(K_i)=K_i \otimes K_i,$$
$$\Delta(E_{ik})=E_{ik} \otimes K_i^{-1} + 1 \otimes E_{ik},
\qquad \Delta(F_{ik})
=F_{ik} \otimes 1 + K_i \otimes F_{ik},$$
$$\epsilon(K_i)=1, \quad \epsilon(E_{ik})=\epsilon(F_{ik})=0,$$
$$S(K_i)=K_i^{-1}, \quad S(E_{ik})=-E_{ik}K_i, \quad S(F_{ik})=
-K_i^{-1}F_{ik}.$$

We have an involution $\tau: \U_v \longrightarrow \U_v$  defined by 
 $E_{ik} \mapsto F_{ik}$, $F_{ik} \mapsto E_{ik}$  
 and $K_i \mapsto K_i^{-1}$  for $i\in I$ and $k=1, \ldots, m_i$.  
We denote by  $\U^0_v$ the subalgebra generated by $K_i^{\pm 1}$, and by $\U^+_v$ (resp. $\U^-_v$) the subalgebra generated by $E_{ik}$ (resp.  $F_{ik}$). Similarly, we define 
$\U^{\geq0}_v$ (resp.  $\U^{\leq0}_v$) to be the subalgebra
generated by $K_i^{\pm 1}$ and $E_{ik}$ (resp.~$K_i^{\pm 1}$ and $F_{ik}$) 
for $i \in I$ and $k=1, \ldots, m_i$. 
Then \[ \U^0_v\isom \CC[K_i^{\pm1}]_{i\in I},\]
and the  involution $\tau$  identifies 
$\U^+_v$ with $\U^-_v$.

The following result provides a triangular decomposition for $\U_v$.

\begin{prop}[\cite{BKM}] \label{tri} 
The multiplication maps 
\[  
\U^0_v \tensor_\CC 
\U^+_v \lra 
\U^{\geq0}_v, 
\qquad
\U^0_v \tensor_\CC 
\U^-_v \lra 
\U^{\leq0}_v, 
\]
and 
\[\U^+_v
\tensor_\CC  \U^0_v
\tensor_\CC \U^-_v
\lra \U_v\]
are isomorphisms of vector spaces.
\end{prop}

\subsection{Embedding of $\U_v^{\ge 0}$ into a Hall algebra}

Let $A=( a_{ij} )_{i,j \in I}$ be a symmetric {Borcherds--Cartan} matrix such that each row has only finitely many nonzero entries and $a_{ii} \neq 0$ for any $i \in I$. 
Fix a  locally
finite quiver $\mathcal Q$ associated to $A$ satisfying conditions (A) and (B): each vertex $i$ has $1-a_{ii}/2$ loops, and two distinct vertices $i$ and $j$  are connected with $-a_{ij}$ arrows for $i \neq j$. Then, since $a_{ii} \neq 0$ for any $i \in I$, the condition (B) is satisfied, and we can always choose an orientation for $\mathcal Q$ so that (A) is satisfied.

If $i \in I^{re}$, then
there exists a unique simple object $S_i \in \mathrm{Rep}_{\k}(\mathcal Q)$ supported at $i$. 
On the other hand, if $i \in I^{im}$
then the set of simple objects supported at $i$ 
is in bijection with $\k^{c_i}$: 
if $\sigma_1, \ldots, \sigma_{c_i}$ denote
the simple loops at $i$ then to $\underline{\lambda}=(\lambda_1, \ldots
,\lambda_{c_i}) \in \k^{c_i}$ corresponds the simple module
$S_{i,\underline{\lambda}}=(V_j,x_{\sigma} )_{j \in I, \sigma \in \Omega}$ with
$\mathrm{dim}_{\k} V_j=
\delta_{ij}$ and $x_{\sigma_k}=\lambda_k \cdot \mathrm{id}$ for $k=1, \ldots, c_i$.

Let us now assume that the charge $\m=(m_i)_{i\in I}$ satisfies
\[ m_i\le |\k^{c_i}|=q^{c_i}   \quad \text{ for each } i \in I.\]
We choose $\underline{\lambda}^{(l)} \in \k^{c_i}$ for $l=1, \ldots,
m_i$ in such a way that 
$\underline{\lambda}^{(l)} \neq \underline{\lambda}^{(l')}$ for $l\neq l'$. 
Then we set $S_{il} = S_{i,\underline{\lambda}^{(l)}}$ for 
$i\in I^{im}$ and $l=1, \dots, m_i$, and simply set $S_{i1}=S_i$ for $i \in I^{re}$. 
Since $S_{il}$ have  the same projectives in the standard resolution \eqref{std} for all $l=1, \dots, m_i$, they define a unique class in $K(\R)$. We will denote this unique class by $\cl{S}_{i}$ for any $i \in I$. 

The following lemma will be used in the proof of Theorem \ref{thm-KS}.

\begin{lem} \label{lem-iso-ka-z}
Let $\bar K (\A)$ denote the image of the map $K(\A) \rightarrow K(\R)$ defined in \eqref{eqn-kr-ka}.  Then the map $\mathbf{dim}:\bar{K}(\A) \rightarrow \ZZ^{\oplus I} $ is well-defined and is an isomorphism. That is,
\[ \ZZ^{\oplus I}  \simeq \bar{K}(\A) \hookrightarrow K(\R) .\]
\end{lem}

\begin{proof}
The map $\mathbf{dim}:\bar{K}(\A) \rightarrow \ZZ^{\oplus I} $ is well-defined  by condition (e) which is verified in Proposition \ref{prop-ae}. Clearly, the map is surjective. As noted already, there is a unique class $\cl{S}_{i}$ which represents all simple modules $S_{il}$ in $\A$ for each $i \in I$. Thus  the map is injective. 
\end{proof}

The following theorem, due to Kang and Schiffmann, is an extension of well-known results of Ringel \cite{Ri} and Green \cite{Gr} from the case of a finite quiver without loops to a locally finite quiver with loops:

\begin{thm}[\cite{KS}] \label{thm-KS}
Suppose $\k=\F$ is such that $|\k^{c_i}| \geq m_i$ for all $i\in I$. 
Then there are injective homomorphisms of algebras 
\[
\U_v^+ 
\lhook \joinrel \longrightarrow 
\H_v (\A) \qquad
\text{ and }
\qquad
\U_v^{\geq 0} 
\lhook \joinrel \longrightarrow 
\widetilde{\H_v}(\R)
\]
defined on generators by $K_{i}^{\pm 1} \mapsto K_{\cl{S}_i}^{\pm 1}$  for $i \in I$, 
$$E_{il} \mapsto
\gen{S_{il}} \cdot (q-1)^{-1} \quad \text{ for }i \in I, $$
where $S_{il}$ and $\cl{S}_i$ are defined right before Lemma \ref{lem-iso-ka-z}.
\end{thm}

\begin{proof}
Let $\widetilde{\H_{v,\k}}(\Q)$ be the extended, twisted Hall algebra defined in \cite{KS}. It is shown in \cite{KS} that  
\[
\U_v^+ 
\lhook \joinrel \longrightarrow 
\H_v (\A) \qquad
\text{ and }
\qquad
\U_v^{\geq 0} 
\lhook \joinrel \longrightarrow 
\widetilde{\H_{v,\k}}(\Q) .
\]
Thus we have only to check that there exists an injective algebra 
homomorphism $\widetilde{\H_{v,\k}}(\Q) \hookrightarrow \widetilde{\H_v}(\R)$. The only difference between the two algebras is that, while $\widetilde{\H_v}(\R)$ is extended by $K(\R)$, the algebra $\widetilde{\H_{v,\k}}(\Q)$ is extended by $\ZZ^{\oplus I}$. Thus the embedding of $\widetilde{\H_{v,\k}}(\Q)$ into $\widetilde{\H_v}(\R)$ follows from Lemma \ref{lem-iso-ka-z}.
\end{proof}

\subsection{Embedding of $\U_v$ into $\DH_\rd(\R)$}

We keep the notations in the previous subsection. In particular, the matrix $A$ is a 
Borcherds--Cartan Matrix and $\mathcal Q$ is a fixed quiver corresponding to $A$. Suppose that $\U_v$ is the quantum group of the  generalized Kac-Moody algebra associated with $A$.    

Now we state and prove the main result of this paper.

\begin{thm} \label{thm-main-q-group}
There is an injective homomorphism of algebra
\[\Xi \colon \U_v\longinto\DH_{\rd}(\R),\]
defined on generators by
\begin{align*}
\Xi (E_{il}) &=(q-1)^{-1}\cdot E_{S_{il}},  &\Xi(F_{il}) &= (-t)\cdot (q-1)^{-1}\cdot {F_{S_{il}}},
\\
\Xi(K_i) &= K_{\cl{S}_{i}}, & \Xi(K_i^{-1}) &= K_{\cl{S}_{i}}^\dag,
\end{align*}
where $l=1, 2, \dots , m_i$ and  $\cl{S}_{i}$ is the unique class representing $S_{il}$ in $K(\R)$ for each $i \in I$. 
\end{thm}

\begin{proof}
We have 
a commutative diagram of linear maps
\[\begin{tikzcd}
\U^+_v\tensor\U^0_v\tensor \U^-_v \ar[rr, "\ \Theta" , hook]\ar[d]& &
{\H_v}(\A)
\tensor\CC[K(\R)]\tensor {\H_v}(\A) \ar[d]\\
\U_v 
 \ar[rr , "\Xi \quad " , hook] && \DH_\rd(\R) 
\end{tikzcd} \]
where the vertical arrows are the isomorphisms described in Propositions \ref{tri2} and \ref{tri}, respectively, and the homomorphism $\Theta$ is constructed out of the homomorphisms of 
Theorem \ref{thm-KS}. The map  $\Xi$ is a well-defined algebra homomorphism by Theorem \ref{thm-KS} and by Lemmas \ref{commute2}, \ref{EF-rel} and Corollary  \ref{cor-associative},  which show that the generators satisfy the defining relations of the quantum group. It is clear from this diagram that $\Xi$ is injective.
\end{proof}

\end{document}